\documentclass[12pt]{article}
\usepackage{amsmath,amsthm, amssymb}
\usepackage{amscd,array}
\usepackage[all]{xy}
\title{Canonical matrices
of bilinear and
sesquilinear forms\footnotetext{This is the authors' version of a work that was accepted for publication in Linear Algebra and its Applications (2007), doi:10.1016/j.laa.2007.07.023.}}

\DeclareMathOperator{\End}{End}
\DeclareMathOperator{\ind}{ind}
\DeclareMathOperator{\charact}{char}
\DeclareMathOperator{\diagg}{diag}

\author{
Roger A. Horn%
\\Department of Mathematics, University of
Utah\\ Salt Lake City,
Utah 84112-0090,
rhorn@math.utah.edu
\and
Vladimir V. Sergeichuk%
\thanks{The research
was done while this
author was visiting
the University of Utah
supported by NSF grant
DMS-0070503 and the
University of S\~ao
Paulo supported by
FAPESP, processo
05/59407-6.}
\\
Institute of
Mathematics,
Tereshchenkivska 3,
Kiev,
Ukraine,\\sergeich@imath.kiev.ua}
\date{}

\begin{document}

\renewcommand{\le}{\leqslant}
\renewcommand{\ge}{\geqslant}

\newcommand{\eprf}{\hfill\hbox{\qedsymbol}}

\newcommand{\diag}{\,\diagdown\,}

\newcommand{\lin}{\,\frac{}{\quad}\,}

\newcommand{\is}{\stackrel
{\text{\raisebox{-1ex}{$\sim\
\;$}}}{\to}}

\newcommand{\ci}{
\begin{picture}(6,6)
\put(3,3){\circle*{3}}
\end{picture}}

\newcommand{\hhat}
{\text{\raisebox{-0.5ex}{\,$\widehat{}$}}\
}

\newcommand{\ddd}{
\text{\begin{picture}(12,8)
\put(-2,-4){$\cdot$}
\put(3,0){$\cdot$}
\put(8,4){$\cdot$}
\end{picture}}}

\newtheorem{theorem}{Theorem}[section]
\newtheorem{lemma}{Lemma}[section]
\newtheorem{corollary}{Corollary}[section]

\theoremstyle{remark}
\newtheorem{remark}{Remark}[section]

\maketitle
\begin{abstract}
Canonical matrices are
given for
\begin{itemize}
  \item
bilinear forms
over an algebraically
closed or real closed
field;  
  \item 
sesquilinear
forms over an
algebraically closed
field and over real
quaternions with any
nonidentity
involution; and
  
  \item 
sesquilinear
forms over a field
$\mathbb F$ of
characteristic
different from $2$
with involution
(possibly, the
identity) up to
classification of
Hermitian forms over
finite extensions of
$\mathbb F$; the
canonical matrices are
based on any given set
of canonical matrices
for similarity over
$\mathbb F$.
  
\end{itemize}
A method for reducing
the problem of
classifying systems of
forms and linear
mappings to the
problem of classifying
systems of linear
mappings is used to
construct the
canonical matrices.
This method has its
origins in
representation theory
and was devised in
[V.V. Sergeichuk, {\it
Math. USSR-Izv.} 31
(1988) 481--501].


 \end{abstract}

\section{Introduction}
\label{s0}

We give canonical
matrices of bilinear
forms over an
algebraically closed
or real closed field
(familiar examples are
$\mathbb C$ and
$\mathbb R$), and of
sesquilinear forms
over an algebraically
closed field and over
$\mathbb
P$-quaternions
($\mathbb P$ is a real
closed field) with
respect to any
nonidentity
involution. We also
give canonical
matrices of
sesquilinear forms
over a field $\mathbb
F$ of characteristic
different from $2$
with involution
(possibly, the
identity) up to
classification of
Hermitian forms over
finite extensions of
$\mathbb F$; the
canonical matrices are
based on any given set
of canonical matrices
for similarity.

Bilinear and
sesquilinear forms
over a field $\mathbb
F$ of characteristic
different from $2$
have been classified
by Gabriel, Riehm, and
Shrader-Frechette.
Gabriel \cite{gab2}
reduced the problem of
classifying bilinear
forms to the
nondegenerate case.
Riehm \cite{rie}
assigned to each
nondegenerate bilinear
form $\cal A\colon
V\times V\to \mathbb
F$ a linear mapping
$A\colon V\to V$ and a
finite sequence
$\varphi^{\cal A}_1,\,
\varphi^{\cal
A}_2,\dots$ consisting
of
$\varepsilon_i$-Hermitian
forms $\varphi^{\cal
A}_i$ over finite
extensions of $\mathbb
F$ and proved that two
nondegenerate bilinear
forms $\cal A$ and
$\cal B$ are
equivalent if and only
if the corresponding
mappings $A$ and $B$
are similar and each
form $\varphi^{\cal
A}_i$ is equivalent to
$\varphi^{\cal B}_i$
(results of this kind
were earlier obtained
by Williamson
\cite{wil}). This
reduction was studied
in \cite{sch} and was
improved and extended
to sesquilinear forms
by Riehm and
Shrader-Frechette
\cite{rie1}. But this
classification of
forms was not
expressed in terms of
canonical matrices, so
it is difficult to
use.

Using Riehm's
reduction, Corbas and
Williams \cite{cor}
obtained canonical
forms of nonsingular
matrices under
congruence over an
algebraically closed
field of
characteristic
different from 2
(their list of
nonsingular canonical
matrices contains an
inaccuracy, which can
be easily fixed; see
\cite[p.\,1013]{hor-ser_can}).
Thompson \cite{thom}
gave canonical pairs
of symmetric or
skew-symmetric
matrices over $\mathbb
C$ and $\mathbb R$
under simultaneous
congruence. Since any
square complex or real
matrix can be
expressed uniquely as
the sum of a symmetric
and a skew-symmetric
matrix, Thompson's
canonical pairs lead
to canonical matrices
for congruence; they
are studied in
\cite{lee-wei}. We
construct canonical
matrices that are much
simpler than the ones
in \cite{cor,lee-wei}.

We construct canonical
matrices of bilinear
and sesquilinear forms
by using the technique
for reducing the
problem of classifying
systems of forms and
linear mappings to the
problem of classifying
systems of linear
mappings that was
devised by Roiter
\cite{roi} and the
second author
\cite{ser_first,
ser_disch,ser_izv}. A
system of forms and
linear mappings
satisfying some
relations is given as
a representation of a
partially ordered
graph $P$ with
relations: each vertex
corresponds to a
vector space, each
arrow or nonoriented
edge corresponds to a
linear mapping or a
bilinear/sesquilinear
form (see Section
\ref{sub_pos1}). The
problem of classifying
such representations
over a field or skew
field $\mathbb F$ of
characteristic
different from $2$
reduces to the
problems of
classifying
\begin{itemize}
\parskip=-3pt
  \item
representations of
some quiver
$\underline P$ with
relations and
involution (in fact,
representations of a
finite dimensional
algebra with
involution) over
$\mathbb F$, and
  \item
Hermitian forms over
fields or skew fields
that are finite
extensions of the
center of $\mathbb F$.
\end{itemize}

The corresponding
reduction theorem was
extended in
\cite{ser_izv} to the
problem of classifying
selfadjoint
representations of a
linear category with
involution and in
\cite{ser_sym} to the
problem of classifying
symmetric
representations of an
algebra with
involution. Similar
theorems were proved
by Quebbermann,
Scharlau, and Schulte
\cite{que-scha,w.schar}
for additive
categories with
quadratic or Hermitian
forms on objects, and
by Derksen, Shmelkin,
and Weyman
\cite{der-wey,shme}
for generalizations of
quivers involving
linear groups.

Canonical matrices of
\begin{itemize}
\item[(i)] bilinear
and sesquilinear
forms,
  \item[(ii)]
pairs of symmetric or
skew-symmetric forms,
and pairs of Hermitian
forms, and
  \item[(iii)]
isometric or
selfadjoint operators
on a space with scalar
product given by a
nondegenerate
symmetric,
skew-symmetric, or
Hermitian form
\end{itemize}
were constructed in
\cite{ser_disch,
ser_izv} by this
technique over a field
$\mathbb F$ of
characteristic
different from 2  up
to classification of
Hermitian forms over
fields that are finite
extensions of $\mathbb
F$. Thus, the
canonical matrices of
(i)--(iii) over
$\mathbb C$ and
$\mathbb R$ follow
from the construction
in \cite{ser_disch,
ser_izv} since
classifications of
Hermitian forms over
these fields are
known.

The canonical matrices
of bilinear and
sesquilinear forms
over an algebraically
closed field of
characteristic
different from 2 and
over a real closed
field given in
\cite[Theorem
3]{ser_izv}, and the
canonical matrices of
bilinear forms over an
algebraically closed
field of
characteristic $2$
given in \cite{ser0}
are based on the
Frobenius canonical
form for similarity.
In this article we
simplify them by using
the Jordan canonical
form. Such a
simplification was
given by the authors
in \cite{hor-ser} for
canonical matrices of
bilinear and
sesquilinear forms
over $\mathbb C$; a
direct proof that the
matrices from
\cite{hor-ser} are
canonical is given in
\cite{hor-ser_regul,
hor-ser_can};
applications of these
canonical matrices
were obtained in
\cite{dok_iso,dok-zha_jor,
dok-zha_rat,
hor-ser_can,
hor-ser_unit}. We also
construct canonical
matrices of
sesquilinear forms
over quaternions; they
were given in
\cite{ser_quat} with
incorrect signs for
the indecomposable
direct summands; see
Remark \ref{rema}.
Analogous results for
canonical matrices of
isometric operators
have been obtained in
\cite{ser_iso}.
\medskip

The paper is organized
as follows. In Section
\ref{s_intre} we
formulate our main
results: Theorem
\ref{t1.1} about
canonical matrices of
bilinear and
sesquilinear forms
over an algebraically
or real closed field
and over quaternions,
and Theorem
\ref{Theorem 5} about
canonical matrices  of
bilinear and
sesquilinear forms
over any field
$\mathbb F$ of
characteristic not $2$
with an involution, up
to classification of
Hermitian forms. In
Section \ref{sub_pos1}
we give a brief
exposition of the
technique for reducing
the problem of
classifying systems of
forms and linear
mappings to the
problem of classifying
systems of linear
mappings. We use it in
Sections \ref{s_pro}
and \ref{secmat}, in
which we prove
Theorems \ref{t1.1}
and \ref{Theorem 5}.

\section{Canonical
matrices for
congruence and
*congruence}
\label{s_intre}

Let $\mathbb F$ be a
field or skew field
with involution
$a\mapsto \bar{a}$,
i.e., a bijection
$\mathbb F\to \mathbb
F$ satisfying
$\overline{a+b}=\bar
a+ \bar b,$
$\overline{ab}=\bar b
\bar a,$ and
$\bar{\bar a}=a.$
Thus, the involution
may be the identity
only if $\mathbb F$ is
a field.

 For any
matrix $A=[a_{ij}]$
over $\mathbb F$, we
write $
A^*:=\bar{A}^T=[\bar{a}_{ji}].
$ Matrices
$A,B\in{\mathbb
F}^{n\times n}$ are
said to be
*{\it\!congruent}\/
over $\mathbb F$ if
there is a nonsingular
$S\in{\mathbb
F}^{n\times n}$ such
that $S^*AS=B$. If
$S^TAS=B$, then the
matrices $A$ and $B$
are called {\it
congruent}. The
transformations of
congruence ($A\mapsto
S^TAS$) and
*congruence ($A\mapsto
S^*AS$) are associated
with the bilinear form
$x^TAy$ and the
sesquilinear form
$x^*Ay$, respectively.

\subsection{Canonical
matrices over an
algebraically or real
closed field and over
quaternions}
\label{sub_01}

In this section we
give canonical
matrices for
congruence over:
\begin{itemize}
  \item
an algebraically
closed field, and

  \item
a {\it real closed
field}---i.e., a field
${\mathbb P}$ whose
algebraic closure
${\mathbb K}$ has a
finite degree $\ne 1$
(that is,
$1<\dim_{\mathbb
P}{\mathbb
K}<\infty$).
\end{itemize}
We also give canonical
matrices for
*congruence over:
\begin{itemize}
  \item
an algebraically
closed field with
nonidentity
involution, and

  \item the {\it skew field
of\/ $\mathbb
P$-quaternions}
\begin{equation*}\label{1ya}
 {\mathbb
 H}=\{a+bi+cj+dk\,|\,a,b,c,d\in\mathbb
 P\},
\end{equation*}
in which $\mathbb P$
is a real closed
field,
$i^2=j^2=k^2=-1$,
$ij=k=-ji,$
$jk=i=-kj,$ and
$ki=j=-ik.$
\end{itemize}
We consider only two
involutions on
$\mathbb H$:
\emph{quaternionic
conjugation}
\begin{equation}\label{ne}
 a+bi+cj+dk\ \longmapsto\
 a-bi-cj-dk, \qquad
a,b,c,d\in\mathbb P,
\end{equation}
and \emph{quaternionic
semiconjugation}
\begin{equation}
\label{nen}
 a+bi+cj+dk\ \longmapsto\
a-bi+cj+dk, \qquad
a,b,c,d\in\mathbb P,
\end{equation}
because if an
involution on $\mathbb
H$ is not quaternionic
conjugation, then it
becomes quaternionic
semiconjugation after
a suitable reselection
of the imaginary units
$i,j,k$; see
\cite[Lemma
2.2]{ser_iso}.

There is a natural
one-to-one
correspondence
\begin{equation*}\label{1ye}
\left\{\parbox{5cm}{$\
$algebraically closed fields\\
with nonidentity
involution}\right\}
\quad\longleftrightarrow
\quad
\bigl\{\text{real
closed fields}\bigr\}
\end{equation*}
sending an
algebraically closed
field with nonidentity
involution to its
fixed field. This
follows from our next
lemma, in which we
collect known results
about such fields.

\begin{lemma}\label{l00}
{\rm(a)} Let\/
$\mathbb P$ be a real
closed field and\/ let
$\mathbb K$ be its
algebraic closure.
Then $\charact{\mathbb
P}=0$ and
\begin{equation}\label{1pp}
\mathbb K={\mathbb
P}+{\mathbb P}i,\qquad
i^2=-1.
\end{equation}
The field\/ ${\mathbb
P}$ has a unique
linear ordering $\le$
such that
\begin{equation*}\label{slr}
\text{$a>0$ and\,
$b>0$}
 \quad\Longrightarrow\quad
\text{$a+b>0$ and\,
$ab>0$}.
\end{equation*}
The positive elements
of\/ $\mathbb P$ with
respect to this
ordering are the
squares of nonzero
elements.

{\rm(b)} Let\/
$\mathbb K$ be an
algebraically closed
field with nonidentity
involution. Then
$\charact\mathbb K=0$,
\begin{equation}\label{123}
\mathbb
P:=\bigl\{k\in{\mathbb
K}\,\bigr|\,
\bar{k}=k\bigr\}
\end{equation}
is a real closed
field,
\begin{equation}\label{1pp11}
\mathbb K={\mathbb
P}+{\mathbb P}i,\qquad
i^2=-1,
\end{equation}
and the involution is
``complex
conjugation'':
\begin{equation}\label{1ii}
\overline{a+bi}=a-bi,\qquad
a,b\in\mathbb P.
\end{equation}

{\rm(c)} Every
algebraically closed
field $\mathbb F$ of
characteristic $0$
contains at least one
real closed subfield.
Hence, $\mathbb F$ can
be represented in the
form \eqref{1pp11} and
possesses the
involution
\eqref{1ii}.
\end{lemma}

\begin{proof} (a)
Let $\mathbb K$ be the
algebraic closure of
$\mathbb F$ and
suppose
$1<\dim_{\mathbb
P}{\mathbb K}<\infty$.
By Corollary 2 in
\cite[Chapter VIII, \S
9]{len}, we have
$\charact{\mathbb
P}=0$ and \eqref{1pp}.
The other statements
of part (a) follow
from Proposition 3 and
Theorem 1 in
\cite[Chapter XI, \S
2]{len}.

(b) If $\mathbb K$ is
an algebraically
closed field with
nonidentity involution
$a\mapsto \bar{a}$,
then this involution
is an automorphism of
order 2. Hence
${\mathbb K}$ has
degree $2$ over its
\emph{fixed field}
${\mathbb P}$ defined
in \eqref{123}. Thus,
${\mathbb P}$ is a
real closed field. Let
$i\in \mathbb K$ be
such that $i^2=-1$. By
(a), every element of
${\mathbb K}$ is
uniquely represented
in the form $k=a+bi$
with $a,b\in{\mathbb
P}$. The involution is
an automorphism of
${\mathbb K}$, so
$\bar{i}^2=-1$. Thus,
$\bar{i}=-i$ and the
involution has the
form \eqref{1ii}.

(c) This statement is
proved in \cite[\S 82,
Theorem 7c]{wan}.
\end{proof}

For notational
convenience, write
\[
A^{-T}:=(A^{-1})^T\quad\text{and}\quad
A^{-*}:=(A^{-1})^*.
\]
The \emph{cosquare} of
a nonsingular matrix
$A$ is $A^{-T}A$. If
two nonsingular
matrices are congruent
then their cosquares
are similar because
\[
(S^TAS)^{-T}(S^TAS)
=S^{-1}A^{-T}AS.
\]
If $\Phi$ is a
cosquare, every matrix
$C$ such that
$C^{-T}C=\Phi$ is
called a {\it cosquare
root} of $\Phi$; we
choose any cosquare
root and denote it by
$\sqrt[T]{\Phi}$.

Analogously, $A^{-*}A$
is the
\emph{*cosquare} of
$A$. If two
nonsingular matrices
are *congruent then
their *cosquares are
similar. If $\Phi$ is
a *cosquare, every
matrix $C$ such that
$C^{-*}C=\Phi$ is
called a {\it
*cosquare root} of
$\Phi$; we choose any
*cosquare root and
denote it by
$\sqrt[\displaystyle
*]{\Phi}$.\label{pppp}

For each real closed
field, we denote by
$\le$ the ordering
from Lemma
\ref{l00}(a). Let
$\mathbb K=\mathbb
P+\mathbb Pi$ be an
algebraically closed
field with nonidentity
involution represented
in the form
\eqref{1pp11}. By the
\emph{absolute value}
of $k=a+bi\in\mathbb
K$ ($a,b\in\mathbb P)$
we mean a unique
nonnegative ``real''
root of $a^2+b^2$,
which we write as
\begin{equation}\label{1kk}
|k|:=\sqrt{a^2+b^2}
\end{equation}
(this definition is
unambiguous since
$\mathbb K$ is
represented in the
form \eqref{1pp11}
uniquely up to
replacement of $i$ by
$-i$). For each
$M\in{\mathbb
K}^{m\times n}$, its
{\it realification}
$M^{\mathbb
P}\in{\mathbb
P}^{2m\times 2n}$ is
obtained by replacing
every entry $a+bi$ of
$M$ by the $2\times 2$
block
\begin{equation}\label{1j}
\begin{matrix}
 a&-b\\b&a
\end{matrix}
\end{equation}

Define the $n$-by-$n$
matrices
\begin{equation*}\label{1aaaa}
\Delta_n(\lambda):=
    \begin{bmatrix}
0&&&\lambda
\\
&&\ddd&i\\
&\lambda&\ddd&\\
\lambda&i&&0
\end{bmatrix},\qquad
J_n(\lambda) :=
\begin{bmatrix}
  \lambda&1&&0\\
  &\lambda&\ddots&\\
  &&\ddots&1\\
0&&&\lambda
\end{bmatrix},
\end{equation*}
\begin{equation*}\label{1aa}
\Gamma_n :=
\begin{bmatrix}
0&&&&&\ddd
\\&&&&1&
\ddd\\
&&&-1&-1&\\ &&1&1&\\
&-1&-1&
&&\\
1&1&&&&0
\end{bmatrix},
\end{equation*}
and
\begin{equation*}\label{1aaa}
\Gamma'_n:=
 \begin{bmatrix}
  0&& & &&-1\\
  && & & \ddd &1\\
  && & -1& \ddd &
   \\
  && 1& 1&& \\
   &\ddd&\ddd&&& \\
  1&1&& & &0
\end{bmatrix}
\text{\quad if $n$ is
even},
 \end{equation*}
\begin{equation*}\label{1aaa2}
\Gamma'_n:=
\begin{bmatrix}
  0&&& &&&1\\
  &&& && \ddd &0\\
  &&&& 1& \ddd &
   \\
  &&& 1&0&&
  \\
  &&1&1 &&& \\
  & \ddd &\ddd& &&&\\
  1&1&& &&&0
\end{bmatrix}
\text{\quad if $n$ is
odd};
 \end{equation*}
the middle groups of
entries are in the
center of $\Gamma'_n$.

The {\it skew sum} of
two matrices $A$ and
$B$ is
\begin{equation*}\label{1.2a}
[A\diag B]:=
\begin{bmatrix}0&B\\A &0
\end{bmatrix}.
\end{equation*}

The main result of
this article is the
following theorem,
which is proved in
Section \ref{secmat}.
It was obtained for
complex matrices in
\cite{hor-ser,
hor-ser_can}.

\begin{theorem}
\label{t1.1} {\rm(a)}
Over an algebraically
closed field of
characteristic
different from $2$,
every square matrix is
congruent to a direct
sum, determined
uniquely up to
permutation of
summands, of matrices
of the form:
   \begin{itemize}
\item [{\rm(i)}]
$J_n(0)$;

\item [{\rm(ii)}]
$[J_n(\lambda)\diag
I_n]$, in which
$\lambda\ne
(-1)^{n+1}$,
$\lambda\ne 0,$ and
$\lambda$ is
determined up to
replacement by
$\lambda^{-1}$;

\item [{\rm(iii)}]
$\sqrt[T]{ J_n
((-1)^{n+1})}$.
\end{itemize}
Instead of the matrix
{\rm(iii)}, one may
use\/ $\Gamma_n$, or\/
$\Gamma'_n$, or any
other nonsingular
matrix whose cosquare
is similar to
$J_n((-1)^{n+1})$;
these matrices are
congruent to
{\rm(iii)}.
\medskip

{\rm (b)} Over an
algebraically closed
field of
characteristic $2$,
every square matrix is
congruent to a direct
sum, determined
uniquely up to
permutation of
summands, of matrices
of the form:
   \begin{itemize}
\item [{\rm(i)}]
$J_n(0)$;

\item [{\rm(ii)}]
$[J_n(\lambda)\diag
I_n]$, in which
$\lambda$ is nonzero
and is determined up
to replacement by
$\lambda^{-1}$;

\item [{\rm(iii)}]
$\sqrt[T]{J_{n}(1)}$
with odd $n$; no
blocks of the form
$[J_n(1)\diag
I_n]$ are permitted
for any odd $n$ for
which a block
$\sqrt[T]{J_{n}(1)}$
occurs in the direct
sum.%
\footnote{If the
direct sum would
otherwise contain both
$\sqrt[T]{J_{n}(1)}$
and $[J_n(1)\diag
I_n]$ for the same odd
$n$, then this pair of
blocks must be
replaced by three
blocks
$\sqrt[T]{J_{n}(1)}$.
This restriction is
imposed to ensure
uniqueness of the
canonical direct sum
because
$\sqrt[T]{J_{n}(1)}
\oplus [J_n(1)\diag
I_n]$ is congruent to
$\sqrt[T]{J_{n}(1)}
\oplus
\sqrt[T]{J_{n}(1)}
\oplus\sqrt[T]{J_{n}(1)}$;
see \cite{ser0} and
Remark \ref{rem3}.}
\end{itemize}
Instead of the matrix
{\rm(iii)}, one may
use\/ $\Gamma'_n$ or
any other nonsingular
matrix whose cosquare
is similar to
$J_n(1)$, these
matrices are congruent
to {\rm(iii)}.
\medskip

{\rm(c)} Over an
algebraically closed
field with nonidentity
involution, every
square matrix is
*congruent to a direct
sum, determined
uniquely up to
permutation of
summands, of matrices
of the form:
   \begin{itemize}
 \item [{\rm(i)}]
$J_n(0)$;

\item [{\rm(ii)}]
$[J_n(\lambda)\diag
I_n]$, in which
$|\lambda|\ne 1$
$($see \eqref{1kk}$)$,
$\lambda\ne 0$, and
$\lambda$ is
determined up to
replacement by
$\bar{\lambda}^{-1}$
$($alternatively, in
which $|\lambda|>1)$;

\item [{\rm(iii)}]
$\pm\sqrt[\displaystyle
*]{ J_n (\lambda)}$,
in which
$|\lambda|=1$.
\end{itemize}
Instead of the
matrices {\rm(iii)},
one may use any of the
matrices
\begin{equation}\label{jde}
\mu\sqrt[\displaystyle
*]{ J_n (1)},\quad
\mu\Gamma_n,\quad
\mu\Gamma_n',\quad
\mu\Delta_n(1),\quad
\mu A
\end{equation}
with $|\mu|=1$, where
$A$ is any $n\times n$
matrix whose *cosquare
is similar to a Jordan
block.
\medskip

{\rm (d)} Over a real
closed field\/
$\mathbb P$ whose
algebraic closure is
represented in the
form \eqref{1pp},
every square matrix is
congruent to a direct
sum, determined
uniquely up to
permutation of
summands, of matrices
of the form:
   \begin{itemize}
\item [{\rm(i)}]
$J_n(0)$;

\item [{\rm(ii)}]
 $[J_n(a)\diag I_n]$,
in which $0\ne a
\in{\mathbb P}$, $a\ne
(-1)^{n+1}$, and $a$
is determined up to
replacement by
$a^{-1}$
$($alternatively, $a
\in{\mathbb P}$ and
$|a|>1$ or $a=
(-1)^{n})$;

 \item
[{\rm(iii)}] $\pm
\sqrt[T]{
J_n((-1)^{n+1})}$;

 \item
[{\rm(ii$'$)}]
$[J_n(\lambda)^{\mathbb
P} \diag I_{2n}]$, in
which
$\lambda\in({\mathbb
P}+{\mathbb
P}i)\smallsetminus
{\mathbb P}$,
$|\lambda|\ne 1$, and
$\lambda$ is
determined up to
replacement by
$\bar{\lambda}$,
$\lambda^{-1}$, or
$\bar{\lambda}^{-1}$
$($alternatively,
$\lambda=a+bi$ with
$a,b\in\mathbb P$,
$b>0$, and
$a^2+b^2>1)$;

 \item
[{\rm(iii$'$)}]
$\pm\sqrt[T]{
J_n(\lambda)^{\,\mathbb
P}}$, in which
$\lambda\in({\mathbb
P}+{\mathbb
P}i)\smallsetminus
{\mathbb P}$,
$|\lambda|=1$, and
$\lambda$ is
determined up to
replacement by
$\bar{\lambda}$
$($alternatively,
$\lambda=a+bi$ with
$a,b\in\mathbb P$,
$b>0$, and
$a^2+b^2=1)$.
\end{itemize}
Instead of {\rm(iii)},
one may use $\pm
\Gamma_n$ or $\pm
\Gamma_n'$.

\noindent Instead of
{\rm(iii$'$)}, one may
use $\pm\big(\sqrt[
\displaystyle *]{
J_n(\lambda)}\,\big)^{\mathbb
P}$ with the same
$\lambda$, or any of
the matrices
\begin{equation}\label{dsk}
\big((c+i)
\Gamma_n\big)^{\mathbb
P},\quad\big((c+i)
\Gamma'_n\big)^{\mathbb
P},\quad\Delta_n(c+i)^{\mathbb
P}
\end{equation}
with $0\ne
c\in{\mathbb P}$.
\medskip

{\rm(e)} Over a skew
field of\/ $\mathbb
P$-quaternions
$($$\mathbb P$ is real
closed$)$ with
quaternionic
conjugation \eqref{ne}
or quaternionic
semiconjugation
\eqref{nen}, every
square matrix is
*congruent to a direct
sum, determined
uniquely up to
permutation of
summands, of matrices
of the form:
\begin{itemize}
\item [{\rm(i)}]
$J_n(0)$;

\item [{\rm(ii)}]
$[J_n(\lambda)\diag
I_n]$, in which
$0\ne\lambda\in\mathbb
P +\mathbb Pi$,
$|\lambda|\ne 1$, and
$\lambda$ is
determined up to
replacement by
$\bar{\lambda}$,
$\lambda^{-1}$, or
$\bar{\lambda}^{-1}$
$($alternatively,
$\lambda=a+bi$ with
$a,b\in\mathbb P$,
$b\ge 0$, and
$a^2+b^2>1)$;

\item [{\rm(iii)}]
$\varepsilon
\sqrt[\displaystyle
*]{ J_n (\lambda)}$,
in which
$\lambda\in\mathbb P
+\mathbb Pi$,
$|\lambda|=1$,
$\lambda$ is
determined up to
replacement by
$\bar{\lambda}$, and
\begin{equation}\label{kki}
\varepsilon :=
  \begin{cases}
     1,&
 \text{if
 the involution is
\eqref{ne},
$\lambda =(-1)^{n}$,}\\
& \text{and if
 the involution is
\eqref{nen},
$\lambda =(-1)^{n+1}$},\\
    \pm 1,&
 \text{otherwise.}
  \end{cases}
\end{equation}
\end{itemize}
Instead of {\rm(iii)},
one may use
\begin{equation}\label{gyo}
 (a+bi)\Gamma_n
 \quad\text{or}\quad
 (a+bi)\Gamma_n',
\end{equation}
in which
$a,b\in\mathbb P$,
$a^2+b^2=1$, and
\[
  \begin{cases}
b\ge 0 & \text{if the
involution is
\eqref{ne}},
\\
  a\ge 0 & \text{if the
involution is
\eqref{nen}}.
  \end{cases}
\]
Instead of {\rm(iii)},
one may also use
\begin{equation}\label{gyo1}
 (a+bi)\Delta_n(1),
\end{equation}
in which
$a,b\in\mathbb P$,
$a^2+b^2=1$, and
\[
  \begin{cases}
a\ge 0, & \text{if the
involution is
\eqref{ne}, $n$ is
even,}
\\
&\text{and if the
involution is
\eqref{nen}, $n$ is
odd},
\\
  b\ge 0, & \text{otherwise}.
  \end{cases}
\]
\end{theorem}

In this theorem
``determined up to
replacement by'' means
that a block is
congruent or
*congruent to the
block obtained by
making the indicated
replacements.

\begin{remark}\label{rem3}
Theorem \ref{tetete}
ensures that each
system of linear
mappings and bilinear
forms on vector spaces
over an algebraically
or real closed field
as well as each system of
linear mappings and
sesquilinear forms on
vector spaces over an
algebraically closed
field or real
quaternions with
nonidentity involution,
decomposes into a
direct sum of
indecomposable systems
that is unique up to
isomorphisms of
summands. Over any
field of
characteristic not
$2$, two
decompositions into
indecomposables may
have nonisomorphic
direct summands, but
Theorem \ref{tetete1}
tells us that the
number of
indecomposable direct
summands does not
depend on the
decomposition.

However, over an
algebraically closed
field $\mathbb F$ of
characteristic $2$,
not even the number of
indecomposable direct
summands is invariant.
For example, the
matrices
\begin{equation}\label{1k}
[\,1\,]\oplus
[\,1\,]\oplus [\,1\,],
\qquad
\begin{bmatrix}
  0 & 1 \\
  1 & 0
\end{bmatrix}
\oplus [\,1\,]
\end{equation}
are congruent over
$\mathbb F$ since
\[
\begin{bmatrix}
  1&0&1\\1&1&0\\1&1&1
\end{bmatrix} \begin{bmatrix}
  1&0&0\\0&1&0\\0&0&1
\end{bmatrix}
\begin{bmatrix}
  1&1&1\\0&1&1\\1&0&1
\end{bmatrix}
=\begin{bmatrix}
  0&1&0\\1&0&0\\0&0&1
\end{bmatrix},
\]
but each of the direct
summands in \eqref{1k}
is indecomposable by
Theorem \ref{t1.1}(b).
The cancellation
theorem does not hold
for bilinear forms
over $\mathbb F$: the
matrices \eqref{1k}
are congruent but the
matrices
\[
[\,1\,]\oplus [\,1\,],
\qquad
\begin{bmatrix}
  0 & 1 \\
  1 & 0
\end{bmatrix}
\]
are not congruent
because they are
canonical.
\end{remark}

\subsection{Canonical matrices
for *congruence over a
field of
characteristic
different from 2}
\label{sub_02}

Canonical matrices for
congruence and
*congruence over a
field of
characteristic
different from 2 were
obtained in
{\cite[Theorem
3]{ser_izv}} up to
classification of
Hermitian forms. They
were based on the
Frobenius canonical
matrices for
similarity. In this
section we rephrase
\cite[Theorem
3]{ser_izv} in terms
of an \emph{arbitrary}
set of canonical
matrices for
similarity. This
flexibility is used in
the proof of Theorem
\ref{t1.1}. The same
flexibility is used in
\cite{hor-ser} to
construct simple
canonical matrices for
congruence or
*congruence over
$\mathbb C$, and in
\cite{ser_iso} to
construct simple
canonical matrices of
pairs $(A,B)$ in which
$B$ is a nondegenerate
Hermitian or
skew-Hermitian form
and $A$ is an
isometric operator
over an algebraically
or real closed field
or over real
quaternions.

In this section
$\mathbb F$ denotes a
field of
characteristic
different from 2 with
involution $a\mapsto
\bar{a}$, which can be
the identity. Thus,
congruence is a
special case of
*congruence.

For each polynomial
\[
f(x)=a_0x^n+a_1x^{n-1}+\dots
+a_n\in \mathbb F[x],
\]
we define the
polynomials
\begin{align*}\label{mau}
\bar f(x)&:=\bar
a_0x^n+\bar
a_1x^{n-1}+\dots+\bar
a_n,\\
f^{\vee}(x)&:=\bar
a_n^{-1}(\bar
a_nx^n+\dots+\bar
a_1x+\bar
a_0)\quad\text{if }
a_n\ne 0.
\end{align*}

The following lemma
was proved in
\cite[Lemma
6]{ser_izv} (or see
\cite[Lemma
2.3]{ser_iso}).

\begin{lemma}
\label{LEMMA 7} Let
$\mathbb F$ be a field
with involution
$a\mapsto \bar a$, let
$p(x) = p^{\vee}(x)$
be an irreducible
polynomial over
$\mathbb F$, and let
$r$ be the integer
part of $(\deg
p(x))/2$. Consider the
field
\begin{equation}\label{alft}
\mathbb F(\kappa) =
\mathbb
F[x]/p(x)\mathbb
F[x],\qquad \kappa:=
x+p(x)\mathbb F[x],
\end{equation}
with involution
\begin{equation}\label{alfta}
f(\kappa)^{\circ} :=
\bar f(\kappa^{-1}).
\end{equation}
Then each element of\/
$\mathbb F(\kappa)$ on
which the involution
acts identically is
uniquely representable
in the form
$q(\kappa)$, in which
\begin{equation}\label{ser13}
q(x)=a_rx^r+\dots+
a_1x +a_0+\bar
a_1x^{-1}+\dots+\bar
a_rx^{-r},
    \quad a_0 = \bar a_0,
\end{equation}
$a_0,\dots
a_r\in\mathbb F;$ if
$\deg p(x) = 2r$ is
even, then
\begin{equation*}\label{uvp}
a_r=
  \begin{cases}
    0
&\text{if the
involution
on $\mathbb F$ is the identity}, \\
    \bar a_r
&\text{if the
involution on $\mathbb
F$ is not the identity
and
$p(0)\ne 1$},\\
    -\bar a_r
&\text{if the
involution on $\mathbb
F$ is not the identity
and $p(0)=1$}.
  \end{cases}
\end{equation*}
\vskip-2em\eprf
\end{lemma}

We say that a square
matrix is
\emph{indecomposable
for similarity} if it
is not similar to a
direct sum of square
matrices of smaller
sizes. Denote by
${\cal O}_{\mathbb F}$
any maximal set of
nonsingular
indecomposable
canonical matrices for
similarity;\label{papage}
this means that each
nonsingular
indecomposable matrix
is similar to exactly
one matrix from ${\cal
O}_{\mathbb F}$.

For example, ${\cal
O}_{\mathbb F}$ may
consist of all
nonsingular {\it
Frobenius blocks},
i.e., the matrices
\begin{equation}\label{3.lfo}
\Phi=\begin{bmatrix}
0&&
0&-c_n\\1&\ddots&&\vdots
\\&\ddots&0&-c_2\\
0&&1& -c_1
\end{bmatrix}
\end{equation}
whose characteristic
polynomials
$\chi_{\Phi}(x)$ are
powers of irreducible
monic polynomials
$p_{\Phi}(x)\ne x$:
\begin{equation}\label{ser24}
\chi_{\Phi}(x)=p_{\Phi}(x)^s
=x^n+
c_1x^{n-1}+\dots+c_n.
\end{equation}
If $\mathbb F$ is an
algebraically closed
field, then we may
take ${\cal
O}_{\mathbb F}$ to be
all nonsingular Jordan
blocks.

It suffices to
construct *cosquare
roots
$\sqrt[\displaystyle
*]{\Phi}$ (see page
\pageref{pppp}) only
for $\Phi\in{\cal
O}_{\mathbb F}$: then
we can take
\begin{equation}\label{ndw}
\sqrt[\displaystyle
*]{\Psi}
=S^*\sqrt[\displaystyle
*]{\Phi}S\qquad\text{if
 $\Psi =S^{-1}\Phi S$
and
$\sqrt[\displaystyle
*]{\Phi}$ exists}
\end{equation}
since  $\Phi=A^{-*}A$
implies $S^{-1}\Phi
S=(S^*AS)^{-*}(S^*AS).
$

Existence conditions
and an explicit form
of
$\sqrt[\displaystyle
*]{\Phi}$ for
Frobenius blocks
$\Phi$ over a field of
characteristic not $2$
were established in
\cite[Theorem
7]{ser_izv}; this
result is presented in
Lemma \ref{lsdy1}. In
the proof of Theorem
\ref{t1.1}, we take
another set ${\cal
O}_{\mathbb F}$ and
construct simpler
*cosquare roots over
an algebraically or
real closed field
$\mathbb F$.

The version of the
following theorem
given in \cite[Theorem
3]{ser_izv} considers
the case in which
${\cal O}_{\mathbb F}$
consists of all
nonsingular Frobenius
blocks.

\begin{theorem}
\label{Theorem 5}
{\rm(a)} Let\/
$\mathbb F$ be a field
of characteristic
different from $2$
with involution
$($which can be the
identity$)$. Let
${\cal O}_{\mathbb F}$
be a maximal set of
nonsingular
indecomposable
canonical matrices for
similarity over
$\mathbb F$. Every
square matrix $A$
over\/ $\mathbb F$ is
*congruent to a direct
sum of matrices of the
following types:
\begin{itemize}
 \item [{\rm(i)}]
$J_n(0)$;

 \item [{\rm(ii)}]
$[\Phi\diag I_n]$, in
which $\Phi\in {\cal
O}_{\mathbb F}$ is an
$n\times n$ matrix
such that
$\sqrt[\displaystyle
*]{\Phi}$ does not
exist $($see Lemma
{\rm\ref{lsdy1}}$)$;
and

 \item [{\rm(iii)}]
$\sqrt[\displaystyle
*]{\Phi}q(\Phi)$, in
which $\Phi\in{\cal
O}_{\mathbb F}$ is
such that
$\sqrt[\displaystyle
*]{\Phi}$ exists and
$q(x)\ne 0$ has the
form \eqref{ser13} in
which $r$ is the
integer part of $(\deg
p_{\Phi}(x))/2$ and
$p_{\Phi}(x)$ is the
irreducible divisor of
the characteristic
polynomial of $\Phi$.
\end{itemize}
The summands are
determined to the
following extent:
\begin{description}
  \item [Type (i)]
uniquely.

  \item [Type (ii)]
up to replacement of
$\Phi$ by the matrix
$\Psi\in{\cal
O}_{\mathbb F}$ that
is similar to
$\Phi^{-*}$ $($i.e.,
whose characteristic
polynomial is
$\chi_{\Phi}^{\vee}(x))$.

  \item [Type (iii)]
up to replacement of
the whole group of
summands
\[
\sqrt[\displaystyle
*]{\Phi}q_1(\Phi)
\oplus\dots\oplus
\sqrt[\displaystyle
*]{\Phi}q_s(\Phi)
\]
with the same $\Phi$
by a direct sum
\[
\sqrt[\displaystyle
*]{\Phi}q'_1(\Phi)
\oplus\dots\oplus
\sqrt[\displaystyle
*]{\Phi}q'_s(\Phi)
\]
such that each
$q'_i(x)$ is a nonzero
function of the form
\eqref{ser13} and the
Hermitian forms
\begin{gather*}\label{777}
q_1(\kappa)x_1^{\circ}x_1+\dots+
q_s(\kappa)x_s^{\circ}x_s,
\\\label{777s}
q'_1(\kappa)x_1^{\circ}x_1+\dots+
q'_s(\kappa)x_s^{\circ}x_s
\end{gather*}
are equivalent over
the field \eqref{alft}
with involution
\eqref{alfta}.
\end{description}

{\rm(b)} In
particular, if\/
$\mathbb F$ is an
algebraically closed
field of
characteristic
different from $2$
with the identity
involution, then the
summands of type
{\rm(iii)} can be
taken equal to
$\sqrt[\displaystyle
*]{\Phi}$. If\/
$\mathbb F$ is an
algebraically closed
field with nonidentity
involution, or a real
closed field, then the
summands of type
{\rm(iii)} can be
taken equal to
$\pm\sqrt[\displaystyle
*]{\Phi}$. In these
cases the summands are
uniquely determined by
the matrix $A$.
\end{theorem}

Let $$f(x)=
\gamma_0x^m +
\gamma_1x^{m-1}+\dots+\gamma_m\in
\mathbb F[x],\qquad
\gamma_0\ne
0\ne\gamma_m.$$ A
vector $(a_1,
a_{2},\dots, a_n)$
over $\mathbb F$ is
said to be
\emph{$f$-recurrent}
if $n\le m$, or if
\[
\gamma_0 a_{l} +
\gamma_{1}a_{l+1}+\dots+
\gamma_ma_{l+m}=0,\qquad
l=1,2,\dots,n - m
\]
(by definition, it is
not $f$-recurrent if
$m=0$). Thus, this
vector is completely
determined by any
fragment of length
$m$.

The following lemma
was stated in
\cite[Theorem
7]{ser_izv} but only a
sketch of the proof
was given.

\begin{lemma}
\label{lsdy1} Let
$\mathbb F$ be a field
of characteristic not
$2$ with involution
$a\mapsto \bar a$
$($possibly, the
identity$)$. Let
$\Phi\in\mathbb
F^{n\times n}$ be
nonsingular and
indecomposable for
similarity; thus, its
characteristic
polynomial is a power
of some irreducible
polynomial
$p_{\Phi}(x)$.

{\rm(a)}
$\sqrt[\displaystyle
*]{\Phi}$ exists if
and only if
\begin{equation}\label{lbdr}
 p_{\Phi}(x) =
p_{\Phi}^{\vee}(x),\ \
\text{and}
\end{equation}
\begin{equation}\label{4.adlw}
\text{if the
involution on $\mathbb
F$ is the identity,
also $p_{\Phi}(x)\ne x
+ (-1)^{n+1}$}.
\end{equation}

{\rm(b)} If
\eqref{lbdr} and
\eqref{4.adlw} are
satisfied and $\Phi$
is a nonsingular
Frobenius block
\eqref{3.lfo} with
characteristic
polynomial
\begin{equation}\label{ser24lk}
\chi_{\Phi}(x)=p_{\Phi}(x)^s
=x^n+
c_1x^{n-1}+\dots+c_n,
\end{equation}
then for
$\sqrt[\displaystyle
*]{\Phi}$ one can take
the Toeplitz matrix
\begin{equation}\label{okjd}
\sqrt[\displaystyle
*]{\Phi}:= [a_{i-j}]=
\begin{bmatrix}
a_0
&a_{-1}&\ddots&a_{1-n}
\\a_{1}&a_0
&\ddots&\ddots
\\\ddots&\ddots&
\ddots&a_{-1}
\\ a_{n-1}&\ddots&
a_{1}&a_0
\end{bmatrix},
\end{equation}
whose vector of
entries
$(a_{1-n},a_{2-n},\dots,a_{n-1})$
is the
$\chi_{\Phi}$-recurrent
extension of the
vector
\begin{equation}\label{ksy}
v=(a_{1-m},\dots,a_{m})
=(a,0,\dots,0,\bar a)
\end{equation}
of length
\begin{equation}\label{leg}
2m=
  \begin{cases}
    n & \text{if $n$ is even}, \\
    n+1 & \text{if $n$ is
odd,}
  \end{cases}
\end{equation}
in which
\begin{equation}
 \label{mag}
a:=
  \begin{cases}
1 & \text{if $n$ is
even,
    except for the
    case}
    \\
    &
\qquad
p_{\Phi}(x)=x+c\
\text{with }c
^{n-1}=-1,
          \\
    \chi_{\Phi}(-1)&
\text{if $n$ is odd
and $p_{\Phi}(x)\ne
x+1$,}
           \\
e-\bar
e&\text{otherwise,
with any fixed $\bar
e\ne e\in\mathbb F$}.
  \end{cases}
\end{equation}
\end{lemma}

\begin{proof} (a)
Let $\Phi\in\mathbb
F^{n\times n}$ be
nonsingular and
indecomposable for
similarity. We prove
here that if
$\sqrt[\displaystyle
*]{\Phi}$ exists then
the conditions
\eqref{lbdr} and
\eqref{4.adlw} are
satisfied; we prove
the converse statement
in (b).

Suppose
$A:=\sqrt[\displaystyle
*]{\Phi}$ exists.
Since
\begin{equation}
 \label{mau3}
A = A^*\Phi=
\Phi^*A\Phi,
\end{equation}
we have $A\Phi
A^{-1}=\Phi^{-*}$ and
\begin{align*}
\chi_{\Phi}(x) &=
\det(xI-\Phi^{-*})=
\det(xI-\bar\Phi^{-1})=
\det((-\bar\Phi^{-1})(I
-x\bar\Phi))=
\\&=\det(-\bar\Phi^{-1})\cdot
x^n\cdot \det(x^{-1}I
-\bar\Phi)=\chi_{\Phi}^{\vee}(x).
\end{align*}
In the notation
\eqref{ser24},
$p_{\Phi}(x)^s =
p_{\Phi}^{\vee}(x)^s$,
which verifies
\eqref{lbdr}.

It remains to prove
\eqref{4.adlw}.
Because of
\eqref{ndw}, we may
assume that $\Phi$ is
a nonsingular
Frobenius block
\eqref{3.lfo} with
characteristic
polynomial
\eqref{ser24lk}. If
$a_{ij}$ are the
entries of $A$, then
we define $a_{i,n+1}$
by $A\Phi= [a_{ij}]
\Phi=[a_{i,j+1}]$,
$a_{n+1,j}$ by
$\Phi^*A\Phi=
\Phi^*[a_{i,j+1}]=
[a_{i+1,j+ 1}]$; and
we then use
\eqref{mau3} to obtain
$[a_{ij}]= [a_{i+1,j+
1}]$. Hence the matrix
entries depend only on
the difference of the
indices and $A$ has
the form \eqref{okjd}
with
$a_{i-j}:=a_{ij}$.
That
$(a_{1-n},a_{2-n},\dots,a_{n-1})$
is
$\chi_{\Phi}$-recurrent
follows from
\begin{equation}\label{fol}
[a_{i-j}]\Phi=[a_{i-j-1}].
\end{equation}

In view of
\begin{equation}\label{lyf1}
\begin{aligned}
&\chi_{\Phi}(x) =x^n+
c_1x^{n-1}+\dots+c_{n-1}x+c_n
\\=&\chi_{\Phi}^{\vee}(x)=
\bar c_n^{-1}(\bar
c_nx^n+ \bar
c_{n-1}x^{n-1}
+\dots+\bar c_1x+1),
\end{aligned}
\end{equation}
the vector $(\bar
a_{n-1},\dots,\bar
a_{1-n})$ is
$\chi_{\Phi}$-recurrent,
so
$
[a_{i-j}]= A= A^*\Phi
= [\bar a_{j-i+1}],
$
and we have
\begin{equation}\label{ser27}
(a_{1-n},\dots,a_{n-1})=
(a_{1-n},\dots,a_0,
\bar a_0,\dots, \bar
a_{2-n}).
\end{equation}
Since this vector is
$\chi_{\Phi}$-recurrent,
it is completely
determined by the
fragment
\begin{equation}\label{ser28}
(a_{1-m},\dots, a_0,
\bar a_0,\dots, \bar
a_{1-m})
\end{equation}
of length $2m$ defined
in \eqref{leg}.

Write
\begin{equation}\label{ser25a}
\mu_{\Phi}(x):=
p_{\Phi}(x)^{s-1}=x^t+
b_1x^{t-1}+\dots+b_t,\qquad
b_0:=1.
\end{equation}

Suppose that
\eqref{4.adlw} is not
satisfied; i.e., the
involution is the
identity and
$p_{\Phi}(x)= x +
(-1)^{n-1}$. Let us
prove that
\begin{equation}\label{kyyyd}
\text{the vector
\eqref{ser28} is
$\mu_{\Phi}(x)$-recurrent.
}
\end{equation}

If $n = 2m$ then
$\mu_{\Phi}(x)
=(x-1)^{2m-1}$ and
\eqref{kyyyd} is
obvious.

Let $n = 2m - 1$. Then
the coefficients of
$\chi_{\Phi}(x)
=(x+1)^{n}$ in
\eqref{ser24lk} and
$\mu_{\Phi}(x)
=(x+1)^{n-1}$ in
\eqref{ser25a} are
binomial coefficients:
\[
c_i=\binom ni,\qquad
b_i=\binom {n-1}i.
\]
Standard identities
for binomial
coefficients ensure
that
\[
c_i=b_i+b_{i-1}=
b_i+b_{n-i},\qquad
0<i<n.
\]
Thus \eqref{kyyyd}
follows since
\begin{equation*} \label{ser29}
\begin{split}
 2[b_0a_{1-m} &+
 b_1a_{2-m}+\dots
 +b_{n-2}a_{3-m}
 +b_{n-1}a_{2-m}]\\
&=(b_0+0)a_{1-m}+
 (b_1+b_{n-1})a_{2-m}
 +
 (b_2+b_{n-2})a_{3-m}\\
 &\qquad +\dots
 +(b_{n-1}+b_1)a_{2-m}
 +(0+b_0)a_{1-m}\\
 &=c_0a_{1-m}+
 c_1a_{2-m}+\dots+
 c_na_{1-m}=0
\end{split}
\end{equation*}
in view of the
$\chi_{\Phi}$-recurrence
of (\ref{ser28}). But
then the
$\mu_{\Phi}$-recurrent
extension of
\eqref{ser28}
coincides with
\eqref{ser27} and we
have
$$(0,\dots,0,b_0,\dots,b_t)
A= 0$$ (see
\eqref{ser25a}), which
contradicts our
assumption that $A$ is
nonsingular.

(b) Let $\Phi$ be a
nonsingular Frobenius
block \eqref{3.lfo}
with characteristic
polynomial
\eqref{ser24lk}
satisfying
\eqref{lbdr} and
\eqref{4.adlw}.

We first prove the
nonsingularity of
every Toeplitz matrix
$A:=[a_{i-j}]$ whose
vector of entries
\begin{equation}\label{oyuyf}
(a_{1-n},a_{2-n},\dots,a_{n-1})
\end{equation}
is
$\chi_{\Phi}$-recurrent
(and so \eqref{fol}
holds) but is not
$\mu_{\Phi}$-recurrent.
If $w:=
(a_{n-1},\dots, a_0)$
is the last row of
$A$, then
\begin{equation}\label{ldyf}
w\Phi^{n-1},\
w\Phi^{n-2},\dots, w
\end{equation}
are all the rows of
$A$ by \eqref{fol}. If
they are linearly
dependent, then
$wf(\Phi) = 0$ for
some nonzero
polynomial $f(x)$ of
degree less than $n$.
If $p_{\Phi}(x)^r$ is
the greatest common
divisor of $f(x)$ and
$\chi_{\Phi}(x)=p_{\Phi}(x)^s$,
then $r<s$ and
$$p_{\Phi}(x)^r=f(x)g(x)+
\chi_{\Phi}(x)h(x)\qquad
\text{for some
}g(x),h(x)\in\mathbb
F[x].$$ Since
$wf(\Phi) = 0$ and
$w\chi_{\Phi}(\Phi) =
0$, we have
$wp_{\Phi}(\Phi)^r =
0$. Thus,
$w\mu_{\Phi}(\Phi) =
0$. Because
\eqref{ldyf} are the
rows of $A$,
\begin{align*}
(0,\dots,0,&b_0,\dots,b_t,
\underbrace{0,\dots,0}_{i})A\\
&=b_0w\Phi^{i+t}+
b_1w\Phi^{i+t-1}+\dots
+b_tw\Phi^{i} =
w\Phi^i\mu_{\Phi}(\Phi)=0
\end{align*}
for each $i=0,1,\dots,
n-t-1$. Hence,
\eqref{oyuyf} is
$\mu_{\Phi}$-recurrent,
a contradiction.

Finally, we must show
that \eqref{ksy} is
$\chi_{\Phi}$-recurrent
but not
$\mu_{\Phi}$-recurrent
(and so in view of
\eqref{lyf1} its
$\chi_{\Phi}$-recurrent
extension has the form
\eqref{ser27}, which
ensures that $A=
[a_{j-i}] = A^*\Phi$
is nonsingular and can
be taken for
$\sqrt[\displaystyle
*]{\Phi}$).

Suppose first that $n
= 2m$. Since
\eqref{ksy} has length
$n$, it suffices to
verify that it is not
$\mu_{\Phi}$-recurrent.
This is obvious if
$\deg \mu_{\Phi}(x)<n
-1$. Let
$\deg\mu_{\Phi}(x)=n
-1$. Then
$\mu_{\Phi}(x)
=(x+c)^{n-1}$ for some
$c$ and we need to
show only that
\begin{equation}\label{nlp}
a + b_{n-1}\bar a = a
+ c^{n-1}\bar a\ne 0.
\end{equation}
If $c ^{n-1}\ne -1$
then by \eqref{mag}
$a=1$ and so
\eqref{nlp} holds. Let
$c ^{n-1}= -1$. If the
involution on $\mathbb
F$ is the identity
then by \eqref{lbdr}
$c=\pm 1$ and so
$c=-1$, contrary to
\eqref{4.adlw}. Hence
the involution is not
the identity,
$a=e-\bar e$, and
\eqref{nlp} is
satisfied.

Now suppose that $n =
2m - 1$. Since
\eqref{ksy} has length
$n + 1$, it suffices
to verify that it is
$\chi_{\Phi}$-recurrent,
i.e., that
\begin{equation}\label{uuv}
a +c_{n}\bar a= 0.
\end{equation}
By \eqref{lyf1}, $c_n
=\bar c_n^{-1}$.
Because
$\chi_{\Phi}(x)
=\chi_{\Phi}^{\vee}(x)
= \bar c_n^{-1}x^n\bar
\chi_{\Phi}(x^{-1}),$
we have
$$\chi_{\Phi}(-1) =
-c_n\overline{
\chi_{\Phi}({-1})}.$$
If $p_{\Phi}(x)\ne
x+1$ then
$a=\chi_{\Phi}(-1)\ne
0$ and \eqref{uuv}
holds. If
$p_{\Phi}(x)= x+1$
then the involution on
$\mathbb F$ is not the
identity by
\eqref{4.adlw}. Hence
$a=e-\bar e$ and
\eqref{uuv} is
satisfied.
\end{proof}

\section{Reduction
theorems for systems
of forms and linear
mappings
}\label{sub_pos1}

Classification
problems for systems
of forms and linear
mappings can be
formulated in terms of
representations of
graphs with
nonoriented, oriented,
and doubly oriented
($\longleftrightarrow$)
edges; the notion of
quiver representations
was extended to such
representations in
\cite{ser_first}. In
this section we give a
brief summary of
definitions and
theorems about such
representations; for
the proofs and a more
detailed exposition we
refer the reader to
\cite{ser_izv} and
\cite{ser_iso}. For
simplicity, we
consider
representations of
graphs without doubly
oriented edges.

Let $\mathbb F$ be a
field or skew field
with involution
$a\mapsto \bar{a}$
(possibly, the
identity). A {\it
sesquilinear form} on
right vector spaces
$U$ and $V$ over
$\mathbb F$ is a
mapping $B\colon
U\times V\to \mathbb
F$ satisfying
\[
  B(ua+u'a',v)=
  \bar{a}B(u,v)+
  \bar{a'}B(u',v)
\]
and
\[
  B(u,va+v'a')
  =B(u,v)a+B(u,v')a'
\]
for all $u,u'\in U$,
$v,v'\in V$, and
$a,a'\in \mathbb F$.
This form is
\emph{bilinear} if the
involution
$a\mapsto\bar{a}$ is
the identity. If
$e_1,\dots,e_m$ and
$f_1,\dots, f_n$ are
bases of $U$ and $V$,
then
$B(u,v)=[u]_e^{*}B_{ef}[v]_f$
for all $u\in U$ and
$v\in V$, in which
$[u]_e$ and $[v]_f$
are the coordinate
vectors and $
B_{ef}:=[B(e_i,f_j)]$
is the matrix of $B$.
Its matrix in other
bases is $R^*B_{ef}S$,
in which $R$ and $S$
are the transition
matrices.

A \emph{pograph}
(partially ordered
graph) is a graph in
which every edge is
nonoriented or
oriented; for example,
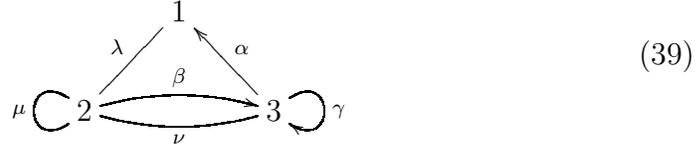
\begin{equation}\label{2.6}
\raisebox{20pt}{\xymatrix{
 &{1}&\\
 {2}\ar@(ul,dl)@{-}_{\mu}
 \ar@{-}[ur]^{\lambda}
 \ar@/^/[rr]^{\beta}
 \ar@/_/@{-}[rr]_{\nu} &&{3}
 \ar[ul]_{\alpha}
 \ar@(ur,dr)^{\gamma}
 }}
\end{equation}
We suppose that the
vertices are
$1,2,\dots,n$, and
that there can be any
number of edges
between any two
vertices.

A {\it representation}
${\cal A}$ of a
pograph $P$ over
$\mathbb F$ is given
by assigning to each
vertex $i$ a right
vector space ${\cal
A}_i$ over $\mathbb
F$, to each arrow
$\alpha\colon i\to j$
a linear mapping
${\cal
A}_{\alpha}\colon
{\cal A}_i\to {\cal
A}_j$, and to each
nonoriented edge
$\lambda\colon i\lin\,
j\ (i\le j)$ a
sesquilinear form
${\cal
A}_{\lambda}\colon
{\cal A}_i\times {\cal
A}_j\to {\mathbb F}$.

For example, each
representation of the
pograph \eqref{2.6} is
a system
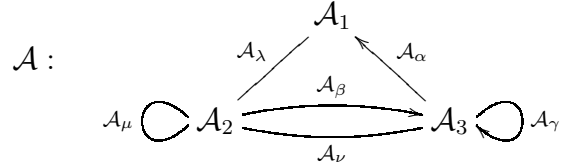
\begin{equation*}\label{2.6aa}
{\cal
A}:\quad\raisebox{20pt}{\xymatrix{
 &{{\cal A}_1}&\\
 {{\cal A}_2}
\save !<-3mm,0cm>
\ar@(ul,dl)
 @{-}_{{\cal A}_{\mu}}
 \restore
\ar@{-}[ur]^{{\cal
A}_{\lambda}}
\ar@/^/[rr]^{{\cal
A}_{\beta}}
\ar@/_/@{-}[rr]_{{\cal
A}_{\nu}} &&{{\cal
A}_3} \ar[ul]_{{\cal
A}_{\alpha}}
 \save
!<3mm,0cm>
\ar@(ur,dr)^{{\cal
A}_{\gamma}}
 \restore
 }}
\end{equation*}
of vector spaces
${\cal A}_1,{\cal
A}_2,{\cal A}_3$ over
$\mathbb F$, linear
mappings ${\cal
A}_{\alpha}$, ${\cal
A}_{\beta}$, ${\cal
A}_{\gamma}$, and
forms $ {\cal
A}_{\lambda}\colon
{\cal A}_1\times {\cal
A}_2\to {\mathbb F}$,
${\cal A}_{\mu}\colon
{\cal A}_2\times {\cal
A}_2\to {\mathbb F}$,
${\cal A}_{\nu}\colon
{\cal A}_2\times {\cal
A}_3\to {\mathbb F}.$

A {\it morphism}
$f=(f_1,\dots,f_n)\colon
{\cal A}\to{\cal A}'$
of representations
${\cal A}$ and ${\cal
A}'$ of $P$ is a set
of linear mappings
$f_i\colon {\cal
A}_i\to {\cal A}'_i$
that transform $\cal
A$ to ${\cal A}'$;
this means that
\[ f_j{\cal
A}_{\alpha}= {\cal
A}'_{\alpha}{f}_i,\qquad
{\cal A}_{\lambda
}(x,y)= {\cal
A}'_{\lambda }
(f_ix,f_jy)\] for all
arrows $ \alpha\colon
i\longrightarrow j$
and nonoriented edges
$\lambda \colon i\lin
j\ (i\le j)$. The
composition of two
morphisms is a
morphism. A morphism
$f\colon{\cal
A}\to{\cal A}'$ is
called an {\it
isomorphism} and is
denoted by $f\colon
{\cal A}\is{\cal A}'$
if all $f_i$  are
bijections. We write
${\cal A}\simeq{\cal
A}'$ if ${\cal A}$ and
${\cal A}'$ are
isomorphic.

The {\it direct sum}\/
${\cal A}\oplus{\cal
A}'$ of
representations ${\cal
A}$ and ${\cal A}'$ of
$P$ is the
representation
consisting of the
vector spaces ${\cal
A}_i\oplus {\cal
A}'_i$, the linear
mappings ${\cal
A}_{\alpha}\oplus
{\cal A}'_{\alpha}$,
and the forms $ {\cal
A}_{\lambda }\oplus
{\cal A}'_{\lambda }$
for all vertices $i$,
arrows $\alpha$, and
nonoriented edges
$\lambda $. A
representation $\cal
A$ is {\it
indecomposable} if
${\cal A}\simeq {\cal
B}\oplus{\cal C}$
implies ${\cal B}=0$
or ${\cal C}=0$, where
$0$ is the
representation in
which all vector
spaces are $0$.

The {\it *dual space}
to a vector space $V$
is the vector space
$V^{*}$ of all
mappings $\varphi:V\to
\mathbb F$ that are
\emph{semilinear},
this means that
$$
\varphi(va+v'a')=\bar{a}(\varphi
v)+ \bar{a'}(\varphi
v'),\qquad v,v'\in V,\
\ a,a'\in \mathbb F.
$$
We identify $V$ with
$V^{**}$ by
identifying $v\in V$
with $\varphi\mapsto
\overline{\varphi v}$.
For every linear
mapping $A:U\to V$, we
define the {\it
*adjoint mapping}
$A^{*}\colon V^{*}\to
U^{*}$ setting
$A^{*}\varphi:=\varphi
A$ for all $\varphi\in
V^*.$

For every pograph
${P}$, we construct
the \emph{quiver
$\underline{P}$ with
an involution on the
set of vertices and an
involution on the set
of arrows} as follows:
we replace
\begin{itemize}
  \item
each vertex $i$ of
${P}$ by two vertices
$i$ and $i^*$,
  \item
each oriented edge
$\alpha\colon i\to j$
by two arrows
$\alpha\colon i\to j$
and $\alpha^*\colon
j^*\to i^*$,
  \item
each nonoriented edge
$\lambda\colon k\lin\,
l\ (k\le l)$ by two
arrows $\alpha\colon
l\to k^*$ and
$\alpha^*\colon k\to
l^*$,
\end{itemize}
and set $u^{**}:=u$
and
$\alpha^{**}:=\alpha$
for all vertices and
arrows of the quiver
$\underline{P}$. For
example,
\begin{equation}\label{4.1}
\raisebox{20pt}{\xymatrix@R=4pt{
 &{2}\ar[dd]_{\alpha}
 \ar@{-}@/^/[dd]^{\lambda}\\
{P}:&
  \\
 &{1} \ar@(ur,dr)@{-}^{\mu}}}
\qquad\qquad\qquad
\raisebox{23pt}{\xymatrix@R=4pt{
 &{2}\ar[dd]_{{\alpha}}
 \ar[ddrr]^(.25){{\lambda}}&
 &{2^*} \\
 {\underline{P}:}&&\\
 &{1}\ar[uurr]^(.75){{\lambda}^*}
 \ar@<0.4ex>[rr]^{{\mu}}
 \ar@<-0.4ex>[rr]_{{\mu}^*}
 &&{1^*}\ar[uu]_{{\alpha}^*}
 }}
\end{equation}

Respectively, for each
representation $\cal
M$ of $P$ over
$\mathbb F$, we define
the representation
$\underline {\cal M}$
of $\underline {P}$ by
replacing
\begin{itemize}
  \item
each vector space $V$
in ${\cal M}$ by the
pair of spaces $V$ and
$V^*$,
  \item
each linear mapping
$A\colon U\to V$ by
the pair of mutually
*adjoint mappings
$A\colon U\to V$ and
$A^*\colon V^*\to
U^*$,

  \item
each sesquilinear form
$B\colon V\times
U\to\mathbb F$ by the
pair of mutually
*adjoint mappings
\begin{equation*}\label{mse}
B\colon u\in U\mapsto
B(?,u)\in V^*,\qquad
B^*\colon v\in
V\mapsto
\overline{B(v,?)}\in
U^*.
\end{equation*}
\end{itemize}

For example, the
following are
representations of
\eqref{4.1}:
\begin{equation}\label{4.2}
\raisebox{23pt}{\xymatrix@R=4pt{
 &{U}\ar[dd]_{A}
 \ar@{-}@/^/[dd]^{B}\\
{\cal A}:&
  \\
 &{V}
\save !<1mm,0cm>
 \ar@(ur,dr)@{-}^{C}
 \restore }}
   \qquad   \qquad   \qquad
\raisebox{23pt}{\xymatrix@R=4pt{
 &{U}\ar[dd]_{A}
 \ar[ddrr]^(.25){B}&
 &{U^{\star}} \\
 {\underline{\cal A}:}&&\\
 &{V}\ar[uurr]^(.75){B^{\star}}
 \ar@<0.4ex>[rr]^{C}
 \ar@<-0.4ex>[rr]_{C^{\star}}
 &&{V^{\star}}\ar[uu]_{A^{\star}}
 }}
\end{equation}

For each
representation $\cal
M$ of $\underline {P}$
we define an
\emph{adjoint
representation} ${\cal
M}^{\circ}$ of
$\underline {P}$
consisting of the
vector spaces ${\cal
M}^{\circ}_v:={\cal
M}^*_{v^*}$ and the
linear mappings ${\cal
M}^{\circ}_{\alpha}:={\cal
M}^*_{\alpha^*}$ for
all vertices $v$ and
arrows $\alpha$ of
$\underline {P}$. For
example, the following
are representations of
the quiver
${\underline{P}}$
defined in
\eqref{4.1}:
\[
\xymatrix@R=4pt{
 &{U_1}\ar[dd]_{A_1}
 \ar[ddrr]^(.25){B_1}&
 &{U_2} \\
 {{\cal M}:}&&\\
 &{V_1}\ar[uurr]^(.75){B_2}
 \ar@<0.4ex>[rr]^{C_1}
 \ar@<-0.4ex>[rr]_{C_2}
 &&{V_2}\ar[uu]_{A_2}
 }
     \qquad\quad
 \xymatrix@R=4pt{
 &{U_2^{\star}}\ar[dd]_{A_2^{\star}}
 \ar[ddrr]^(.25){B_2^{\star}}&
 &{U_1^{\star}} \\
 {{\cal M}^{\circ}:}&&\\
 &{V_2^{\star}}
 \ar[uurr]^(.75){B_1^{\star}}
 \ar@<0.4ex>[rr]^{C_2^{\star}}
 \ar@<-0.4ex>[rr]_{C_1^{\star}}
 &&{V_1^{\star}}\ar[uu]_{A_1^{\star}}
 }
\]
The second
representation in
\eqref{4.2} is
\emph{selfadjoint}:
$\underline{\cal
A}^{\circ}
=\underline{\cal A}$.

In a similar way, for
each morphism $f\colon
{\cal M}\to{\cal N}$
of representations of
$\underline{P}$ we
construct the {\it
adjoint morphism}
\begin{equation}\label{kdtc}
f^{\circ}\colon {\cal
N}^{\circ}\to{\cal
M}^{\circ},\qquad
\text{in which }\
f^{\circ}_i:=f^*_{i^*}
\end{equation}
for all vertices $i$
of $\underline {P}$.
An isomorphism
$f\colon {\cal
M}\is{\cal N}$ of
selfadjoint
representations ${\cal
M}$ and ${\cal N}$ is
called a {\it
congruence} if
$f^{\circ}=f^{-1}$.

For each isomorphism
$f\colon {\cal
A}\is{\cal B}$ of
representations of a
pograph ${P}$, we
define the
\emph{congruence}
$\underline{f}\colon
\underline{\cal A}\is
\underline{\cal B}$ of
the corresponding
selfadjoint
representations of
$\underline{P}$ by
defining:
\begin{equation*}\label{ldi}
\underline{f}_{\,i}:=f_i,
\quad
\underline{f}_{\,i^*}:=f_i^{-*}
\quad\text{for each
vertex $i$ of $P$.}
\end{equation*}

Two representations
$\cal A$ and $\cal B$
of a pograph $P$ are
isomorphic if and only
if the corresponding
selfadjoint
representations
$\underline{\cal A}$
and $\underline{\cal
B}$ of the quiver
$\underline{P}$ are
congruent. Therefore,
\emph{the problem of
classifying
representations of a
pograph $P$ up to
isomorphism reduces to
the problem of
classifying
selfadjoint
representations of the
quiver $\underline{P}$
up to congruence.}

Let us show how to
solve the latter
problem if we know a
maximal set $\ind
(\underline{P})$ of
nonisomorphic
indecomposable
representations of the
quiver $\underline{P}$
(this means that every
indecomposable
representation of
$\underline{P}$ is
isomorphic to exactly
one representation
from $\ind
(\underline{P})$). We
first replace each
representation in
$\ind (\underline{P})$
that is {\it
isomorphic} to a
selfadjoint
representation by one
that is {\it actually}
selfadjoint---i.e.,
has the form
$\underline{\cal A}$,
and denote the set of
these $\underline{\cal
A}$ by
$\ind_0(\underline{P})$.
Then in each of the
one- or two-element
subsets
$$
\{{\cal M},{\cal
L}\}\subset\ind
(\underline{P})
\smallsetminus
\ind_0(\underline{P})
\quad \text{such that
}{\cal
M}^{\circ}\simeq {\cal
L},
$$
we select one
representation and
denote the set of
selected
representations by
$\ind_1(\underline{P})$.
We obtain a new set
$\ind (\underline{P})$
that we partition into
3 subsets:
\begin{equation}\label{4.8d}
{\ind (\underline{P})}
=
\begin{tabular}{|c|c|}
 \hline &\\[-12pt]
  $\; {\cal M}\;  $&
  ${\cal M}^{\circ}\text{ (if
  ${\cal M}^{\circ}\not
  \simeq{\cal M})$}$\\
  \hline
\multicolumn{2}{|c|}
{$\underline{\cal A}\vphantom{{\hat{N}}}$}\\
 \hline
\end{tabular}\,,\;
\begin{matrix}
 {\cal M}\in
\ind_1(\underline{P}),\\[1pt]
\underline{\cal
A}\in\ind_0(\underline{P}).
\end{matrix}
\end{equation}

For each
representation ${\cal
M}$ of
$\underline{P}$, we
define a
representation ${\cal
M}^+$ of $P$ by
setting $ {\cal
M}^+_i:={\cal
M}_i\oplus {\cal
M}_{i^*}^*$ for all
vertices $i$ in $P$
and
\begin{equation}\label{piyf}
{\cal M}^+_{\alpha }:=
\begin{bmatrix} {\cal
M}_{\alpha}&0\\0&{\cal
M}_{\alpha^*}^{*}
 \end{bmatrix},\qquad
{\cal M}^+_{\beta}:=
 \begin{bmatrix}
   0&{\cal
M}_{\beta^*}^{*}\\{\cal
M}_{\beta}&0
 \end{bmatrix}
\end{equation}
for all edges $
\alpha\colon
i\longrightarrow j$
and $\beta\colon i\lin
j\ (i\le j)$. The
representation ${\cal
M}^+$ arises as
follows: each
representation ${\cal
M}$ of $\underline P$
defines the
selfadjoint
representation ${\cal
M}\oplus {\cal
M}^{\circ}$; the
corresponding
representation of $P$
is ${\cal M}^+$ (and
so $\underline{\cal
M}^+={\cal M}\oplus
{\cal M}^{\circ}$).

For every
representation ${\cal
A}$ of ${P}$ and for
every selfadjoint
automorphism
$f=f^{\circ}\colon
\underline{\cal
A}\is\underline{\cal
A}$, we denote by
${\cal A}^f$ the
representation of $P$
that is obtained from
${\cal A}$ by
replacing each form
${\cal A}_{\beta}$
$(\beta\colon i\lin
j$, $i\le j)$ by
${\cal
A}^f_{\beta}:={\cal
A}_{\beta}f_j$.

Let ${\ind
(\underline{P})}$ be
partitioned as in
\eqref{4.8d}, and let
$\underline{\cal
A}\in{\ind_0
(\underline{P})}$. By
\cite[Lemma
1]{ser_izv}, the set
$R$ of noninvertible
elements of the
endomorphism ring
$\End (\underline{\cal
A})$ is the radical.
Therefore, $\mathbb
T({\cal A}):=\End
(\underline{\cal
A})/R$ is a field or
skew field, on which
we define the
involution
\begin{equation}\label{kyg}
(f+R)^{\circ}:=f^{\circ}+R.
\end{equation}

For each nonzero
$a=a^{\circ}\in
\mathbb T({\cal A})$,
we fix a selfadjoint
automorphism
\begin{equation}\label{lfw}
f_a=f_a^{\circ}\in
a,\quad\text{ and
define ${\cal
A}^a:={\cal A}^{f_a}$}
\end{equation}
(we can take
$f_a:=(f+f^{\circ})/2$
for any $f\in a$). The
set of representations
${\cal A}^a$ is called
the \emph{orbit of}
${\cal A}$.

For each Hermitian
form
\[
\varphi(x)=x^{\circ}_1a_1x_1+\dots+
x^{\circ}_ra_rx_r,\qquad
0\ne
a_i=a_i^{\circ}\in
\mathbb T({\cal A}),
\]
we write
\[
{\cal
A}^{\varphi(x)}:=
{\cal
A}^{a_1}\oplus\dots\oplus
{\cal A}^{a_r}.
\]

The following theorem
is a special case of
\cite[Theorem
1]{ser_izv} (or
\cite[Theorem
3.1]{ser_iso}).
\begin{theorem}
\label{tetete1} Over a
field or skew field\/
$\mathbb F$ of
characteristic
different from $2$
with involution
$a\mapsto \bar{a}$
$($possibly, the
identity$)$, every
representation of a
pograph $P$ is
isomorphic to a direct
sum
\begin{equation*}\label{iap}
{\cal
M}_1^+\oplus\dots\oplus
{\cal M}_p^+\oplus
{\cal
A}_1^{\varphi_1(x)}\oplus
\dots\oplus {\cal
A}_q^{\varphi_q(x)},
\end{equation*}
in which
\[
{\cal M}_i\in
\ind_1(\underline{P}),\qquad
\underline{\cal
A}_j\in
\ind_0(\underline{P}),
\]
and ${\cal A}_j\ne
{\cal A}_{j'}$ if
$j\ne j'$. This sum is
determined by the
original
representation
uniquely up to
permutation of
summands and
replacement of ${\cal
A}_j^{\varphi_j(x)}$
by ${\cal
A}_j^{\psi_j(x)}$, in
which ${\varphi_j(x)}$
and ${\psi_j(x)}$ are
equivalent Hermitian
forms over $\mathbb
T({\cal A}_j)$ with
involution
\eqref{kyg}. \eprf
\end{theorem}

Theorem \ref{tetete1}
implies the following
generalization of the
law of inertia for
quadratic forms.

\begin{theorem}[{\cite%
[Theorem
3.2]{ser_iso}}]
 \label{tetete}
Let $\mathbb F$ be
either
\begin{itemize}
  \item[\rm(i)]
an algebraically
closed field of
characteristic
different from $2$
with the identity
involution, or

  \item[\rm(ii)]
an algebraically
closed field with
nonidentity
involution, or

  \item[\rm(iii)]
a real closed field,
or the skew field of
quaternions over a
real closed field.
\end{itemize}

Then every
representation of a
pograph $P$ over
$\mathbb F$ is
isomorphic to a direct
sum, uniquely
determined up to
permutation of
summands, of
representations of the
types:
\begin{equation}\label{do}
{\cal M}^+,\
  \begin{cases}
   {\cal
A}& \text{if ${\cal
A}^{-}\simeq{\cal
A}$}, \\
{\cal A},\ {\cal
A}^{-} & \text{if
${\cal
A}^{-}\not\simeq{\cal
A}$},
  \end{cases}
\end{equation}
in which ${\cal M}\in
\ind_1(\underline{P})$
and $\underline{\cal
A}\in\ind_0(\underline{P})$.
In the respective
cases {\rm(i)--(iii)},
the representations
\eqref{do} have the
form
\begin{itemize}
  \item[\rm(i)]
${\cal M}^+$, ${\cal
A}$,

  \item[\rm(ii)]
${\cal M}^+$, ${\cal
A}$, ${\cal A}^{-}$,

\item[\rm(iii)] $
{\cal M}^+,
  \begin{cases}
    \ \ {\cal A}, &
\parbox[t]{270pt}{if $\mathbb
T({\cal A})$ is an
algebraically closed
field with the
identity involution or
a skew field of
quaternions with
involution different
from quaternionic
conjugation,} \\
    {\cal A},{\cal A}^{-}, &
    \text{otherwise}.
  \end{cases}
$
\end{itemize}
\end{theorem}

\begin{remark}\label{rema}
Theorem \ref{tetete}
is a special case of
Theorem 2 in
\cite{ser_izv}, which
was formulated
incorrectly in the
case of quaternions.
To correct it, remove
``or the algebra of
quaternions \dots'' in
a) and b) and add ``or
the algebra of
quaternions over a
maximal ordered
field'' in c). The
paper \cite{ser_quat}
is based on the
incorrect Theorem 2 in
\cite{ser_izv} and so
the signs $\pm$ of the
sesquilinear forms in
the indecomposable
direct summands in
\cite[Theorems
1--4]{ser_quat} are
incorrect. Correct
canonical forms are
given for
bilinear/sesquilinear
forms in Theorem
\ref{t1.1}, for pairs
of
symmetric/skew-symmetric
matrices in
\cite{rod_pair_nonst,rod_pair_stand},
for selfadjoint
operators in
\cite{kar}, and for
isometries in
\cite{ser_iso}.
\end{remark}

\section{Proof of Theorem
\ref{Theorem 5} }
\label{s_pro}

Each sesquilinear form
defines a
representation of the
pograph
\begin{equation}\label{jsos+}
\xymatrix{ P\,: &{1}
 \ar@(ur,dr)@{-}^{\alpha}
 }
\end{equation}
Its quiver is
\begin{equation*}\label{jso+}
\underline{
P}:\quad\xymatrix{
 {1}
  \ar@/^/@{->}[rr]^{\alpha}
 \ar@/_/@{->}[rr]_{\alpha^*} &&{1^*}
 }
\end{equation*}
We prove Theorem
\ref{Theorem 5} using
Theorem \ref{tetete1};
to do this, we first
identify in Lemma
\ref{lenhi} the sets
$\ind_1(\underline{P})$
and
$\ind_0(\underline{P})$,
and the orbit of $\cal
A$ for each
$\underline{\cal A}\in
\ind_0(\underline{P})$.

Every representation
of $P$ or $\underline
P$ over $\mathbb F$ is
isomorphic to a
representation in
which all vector
spaces are $\mathbb
F\oplus\dots\oplus\mathbb
F$. From now on, we
consider only such
representations of $P$
and $\underline P$;
they can be given by a
square matrix $A$:
\begin{equation}\label{ren7a+}
{\cal A}:\quad
\xymatrix{
 *{\ci}
 \ar@(ur,dr)@{-}^{
A}}\qquad\text{(we
write ${\cal A}=A$)}
 \end{equation}
and, respectively, by
rectangular matrices
$A$ and $B$ of the
same size:
\begin{equation*}\label{ser14sd}
{\cal
M}:\quad\xymatrix{
 {\ci}
\ar@/^/@{->}[rr]^{A}
\ar@/_/@{->}[rr]_{B}
&&{\ci}
 }\qquad\text{(we
write ${\cal
M}=(A,B)$),}
    \end{equation*}
we omit the spaces
$\mathbb
F\oplus\dots\oplus\mathbb
F$ since they are
completely determined
by the sizes of the
matrices.

The adjoint
representation
\begin{equation*}\label{ren2+}
{\cal
M}^{\circ}:\quad\xymatrix{
 {\ci}
\ar@/^/@{->}[rr]^{B^*}
\ar@/_/@{->}[rr]_{A^*}
&&{\ci}
 }
 \end{equation*}
is given by the matrix
pair
\begin{equation}\label{ldbye}
(A,B)^{\circ} =
(B^*,A^*).
\end{equation}

A morphism of
representations
\begin{equation*}\label{ser14d}
\xymatrix@R=12pt{{\
{\cal M}:}\ar[dd]_f&&
 *{\ci}\ar[dd]_{F_1}
  \ar@/^/@{->}[rr]^{A}
 \ar@/_/@{->}[rr]_{B}
 &&*{\ci}\ar[dd]^{F_2}
\\ \\
 {\ {\cal M}':}&&*{\ci}
  \ar@/^/@{->}[rr]^{A'}
 \ar@/_/@{->}[rr]_{B'}
 &&*{\ci}
 }
\end{equation*}
is given by the matrix
pair
$f=[F_1,F_2]\colon\;
{\cal M}\to {\cal M}'$
(for morphisms we use
square brackets)
satisfying
\begin{equation}\label{msp}
F_2A=A'F_1,\qquad
F_2B=B'F_1.
\end{equation}

Denote by $0_{m0}$ and
$0_{0n}$ the $m\times
0$ and $0\times n$
matrices representing
the linear mappings
$0\to {\mathbb F}^m$
and ${\mathbb F}^n\to
0$. Thus,
$0_{m0}\oplus 0_{0n}$
is the $m\times n$
zero matrix.

\begin{lemma}\label{lenhi}
Let $\mathbb F$ be a
field or skew field of
characteristic
different from $2$.
Let ${\cal O}_{\mathbb
F}$ be a maximal set
of nonsingular
indecomposable
canonical matrices
over\/ $\mathbb F$ for
similarity. Let $P$ be
the pograph
\eqref{jsos+}. Then:

\begin{itemize}
  \item[{\rm(a)}] The set
$\ind(\underline{P})$
can be taken to be the
set of all
representations
\begin{equation*}\label{ser16}
(\Phi, I_n),\ (J_n(0),
I_n),\, (I_n,J_n(0)),\
(M_n,N_n),\
(N_n^T,M_n^T)
\end{equation*}
in which $\Phi\in
{\cal O}_{\mathbb F}$
is $n$-by-$n$ and
\begin{equation*}\label{ser17}
M_n:=\begin{bmatrix}
1&0&&0\\&\ddots&\ddots&\\0&&1&0
\end{bmatrix},\quad
N_n:=\begin{bmatrix}
0&1&&0\\&\ddots&\ddots&\\0&&0&1
\end{bmatrix}
\end{equation*}
are $(n-1)$-by-$n$ for
each natural number
$n$.

\item[\rm(b)] The set
$\ind_1(\underline{P})$
can be taken to be the
set of all
representations
\begin{equation*}\label{sew}
(\Phi, I_n),\quad
(J_n(0), I_n),\quad
(M_n,N_n)
\end{equation*}
in which $\Phi\in{\cal
O}_{\mathbb F}$  is an
$n\times n$ matrix
such that
$\sqrt[\displaystyle
*]{\Phi}$ does not
exist, and
\begin{equation}\label{4.adg}
\parbox{18em}
{$\Phi$ is determined
up to replacement by\\
the unique
$\Psi\in{\cal
O}_{\mathbb F}$ that
is similar to
$\Phi^{-*}$.}
\end{equation}
The corresponding
representations of
${P}$ are
\begin{equation}\label{dwh}
(\Phi,
I_n)^+=[\Phi\diag
I_n],
\end{equation}
\begin{equation}\label{dld}
(M_n,N_n)^+\simeq
J_{2n-1}(0),\qquad
(J_n(0), I_n)^+\simeq
J_{2n}(0).
\end{equation}

  \item[\rm(c)]
The set
$\ind_0(\underline{{P}})$
can be taken to be the
set of all
representations
\begin{equation}\label{nsp}
\underline{\cal
A}_{\Phi} :=
(\sqrt[\displaystyle
*]{\Phi},
(\sqrt[\displaystyle
*]{\Phi})^*)
\end{equation}
in which $\Phi\in{\cal
O}_{\mathbb F}$ is
such that
$\sqrt[\displaystyle
*]{\Phi}$ exists. The
corresponding
representations of
${P}$ are
\begin{equation}\label{5mi}
{\cal A}_{\Phi}
=\sqrt[\displaystyle
*]{\Phi},\quad {\cal
A}_{\Phi}^-
=-\sqrt[\displaystyle
*]{\Phi}, \quad {\cal
A}_{\Phi}^f=
\sqrt[\displaystyle
*]{\Phi}F,
\end{equation}
in which
$f=[F,F^*]\colon
\underline{\cal
A}_{\Phi}\is
\underline{\cal
A}_{\Phi}$ is a
selfadjoint
automorphism.

\item[\rm(d)] Let
$\mathbb F$ be a field
and let
$\underline{\cal
A}_{\Phi} :=
(\sqrt[\displaystyle
*]{\Phi},
(\sqrt[\displaystyle
*]{\Phi})^*)\in
\ind_0(\underline{{P}})$,
in which $\Phi$ is a
nonsingular matrix
over $\mathbb F$ that
is indecomposable for
similarity $($thus,
its characteristic
polynomial is a power
of some irreducible
polynomial
$p_{\Phi})$.
\begin{itemize}
  \item[\rm(i)]
The ring
$\End(\underline{\cal
A}_{\Phi})$ of
endomorphisms of
$\underline{\cal
A}_{\Phi}$ consists of
the matrix pairs
\begin{equation}\label{ldy}
[f(\Phi),f(\Phi^{-*})],\qquad
f(x)\in\mathbb F[x],
\end{equation}
and the involution on
$\End(\underline{\cal
A}_{\Phi})$ is
\[
[f(\Phi),f(\Phi^{-*})]^{\circ}=
[\bar
f(\Phi^{-1}),\bar
f(\Phi^*)].
\]

  \item[\rm(ii)]
$\mathbb T({\cal
A}_{\Phi})$ can be
identified with the
field
\begin{equation}\label{ksyq}
{\mathbb
F}(\kappa)={\mathbb
F}[x]/p_{\Phi}(x){\mathbb
F}[x],\qquad
\kappa:=x+p_{\Phi}(x){\mathbb
F}[x],
\end{equation}
with involution
\begin{equation}\label{kxu}
f(\kappa)^{\circ}=
\bar f(\kappa^{-1}).
\end{equation}
Each element of
$\mathbb T({\cal
A}_{\Phi})$ on which
this involution acts
identically is
uniquely represented
in the form
$q(\kappa)$ for some
nonzero function
\eqref{ser13}. The
representations
\begin{equation}\label{wmp}
{\cal
A}_{\Phi}^{q(\kappa)}:
\quad \ \xymatrix{
 *{\ci}
 \ar@(ur,dr)@{-}^{
\sqrt[\displaystyle
*]{\Phi}\,q(\Phi)}}
\end{equation}
$($see \eqref{lfw}$)$
constitute the orbit
of ${\cal A}_{\Phi}$.
\end{itemize}
\end{itemize}

\end{lemma}

\begin{proof}
(a) This form of
Kronecker's theorem
about matrix pencils
follows from
\cite[Sect.
11.1]{gab-roi}.

(b)\,\&\,(c) Let
$\Phi,\Psi\in{\cal
O}_{\mathbb F}$ be
$n$-by-$n$. In view of
\eqref{ldbye},
$(\Phi,I_n)^{\circ}
=(I_n,\Phi^*)\simeq
(\Phi^{-*},I_n)$ and
so
\begin{equation}\label{mdpiug}
(\Psi,I_n)\simeq
(\Phi,I_n)^{\circ}
\quad\Longleftrightarrow\quad
\text{$\Psi$ is
similar to
$\Phi^{-*}$.}
\end{equation}

Suppose $(\Phi,I_n)$
is isomorphic to a
selfadjoint
representation:
\begin{equation}\label{bgd}
[F_1,F_2]\colon\;
(\Phi,I_n) \is
(B,B^*).
\end{equation}
Define a selfadjoint
representation
$(A,A^*)$ by the
congruence
\begin{equation}\label{wfw}
[F_1^{-1},F_1^*]\colon\;
(B,B^*) \is (A,A^*).
\end{equation}
The composition of
\eqref{bgd} and
\eqref{wfw} is the
isomorphism
\[
\xymatrix@R=15pt{
 *{\ci}\ar[dd]_{I_n}
  \ar@/^/@{->}[rr]^{\Phi}
 \ar@/_/@{->}[rr]_{I_n}
 &&*{\ci}\ar[dd]^{F:=F^*_1F_2}
\\ \\
 *{\ci}
  \ar@/^/@{->}[rr]^{A}
 \ar@/_/@{->}[rr]_{A^*}
 &&*{\ci}
 }
\]
By \eqref{msp}, $ A =
F\Phi$ and $A^* = F$.
Thus $A = A^*\Phi$.
Taking
$A=\sqrt[\displaystyle
*]{\Phi}$, we obtain
\begin{equation*}\label{kdi}
[I_n,(\sqrt[\displaystyle
*]{\Phi})^*]\colon
(\Phi,I_n)\is
(\sqrt[\displaystyle
*]{\Phi},(\sqrt[\displaystyle
*]{\Phi})^*).
\end{equation*}

This means that if
$(\Phi,I_n)\in
\ind(\underline{{P}})$
is isomorphic to a
selfadjoint
representation, then
$(\Phi,I_n)$ is
isomorphic to
\eqref{nsp}. Hence,
the representations
\eqref{nsp} comprise
$\ind_0(\underline{{P}})$.
Due to \eqref{mdpiug},
we can identify
isomorphic
representations in the
set of remaining
representations
$(\Phi,I_n)\in
\ind(\underline{{P}})$
by imposing the
condition
\eqref{4.adg}; we then
obtain
$\ind_1(\underline{P})$
from Lemma
\ref{lenhi}(b).

To verify \eqref{dld},
we prove that $J_m(0)$
is permutationally
similar to
\[
  \begin{cases}
  (M_n,N_n)^+=[M_n\diag
N_n^T] & \text{if $m=2n-1$}, \\
    (J_n(0),
I_n)^+=[J_n(0)\diag
I_n] & \text{if
$m=2n$}
  \end{cases}
\]
(see \eqref{piyf}).
The units of $J_m(0)$
are at the positions $
(1,2)$,
$(2,3),\,\dots,\,$$(m-1,m);
$ so it suffices to
prove that there is a
permutation $f$ on
$\{1,2,\dots,m\}$ such
that
\[
(f(1),f(2)),\ \
(f(2),f(3)),\ \dots,\
(f(m-1),f(m))
\]
are the positions of
the unit entries in
$[M_n\diag N_n^T]$ or
$[J_n(0)\diag I_n]$.
This becomes clear if
we arrange the
positions of the unit
entries in the
$(2n-1)\times(2n-1)$
matrix
$$
[M_n\diag
N_n^T]=\left[
\begin{array}{c|c}
\text{\large\rm 0}&
\begin{matrix}
0&&0\\1&\ddots&\\
&\ddots&0\\0&&1
\end{matrix}
\\ \hline
\begin{matrix}
1&0&&0\\&\ddots&\ddots&\\
0&&1&0
\end{matrix}
& \text{\large\rm 0}
\end{array}\right]
$$
as follows:
\begin{multline*}
(n,2n-1),\,
(2n-1,n-1),\,
(n-1,2n-2),\\
(2n-2,n-2),\dots,
(2,n+1),\, (n+1,1),
\end{multline*}
and the positions of
the unit entries in
the $2n\times 2n$
matrix $[J_n(0)\diag
I_n]$ as follows:
\[
(1,n+1),\,(n+1,2),\,
(2,n+2),\, (n+2,3),
\dots,(2n-1,n),\,(n,2n).
\]

(d) Let $\mathbb F$ be
a field.  If $\Phi$ is
a square matrix over
$\mathbb F$ that is
indecomposable for
similarity, then each
matrix over $\mathbb
F$ that commutes with
$\Phi$ is a polynomial
in $\Phi$. To verify
this, we may assume
that $\Phi$ is an
$n\times n$ Frobenius
block \eqref{3.lfo}.
Then the vectors
\begin{equation}\label{gwz}
e:=(1,0,\dots,0)^T,\
\Phi e,\ \dots,\
\Phi^{n-1} e
\end{equation}
 form a
basis of $\mathbb
F^n$. Let $S\in\mathbb
F^{n\times n}$ commute
with $\Phi$, let
\[S e=a_0e+a_1\Phi
e+\dots+
a_{n-1}\Phi^{n-1}e,\qquad
a_0,\dots,
a_{n-1}\in\mathbb F,\]
and let
$f(x):=a_0+a_1x
+\dots+
a_{n-1}x^{n-1}\in\mathbb
F[x].$ Then
$Se=f(\Phi)e$ and
\[S\Phi e=\Phi Se=\Phi f(\Phi)e=
f(\Phi)\Phi e, \
\dots, \ S\Phi^{n-1}
e=f(\Phi)\Phi^{n-1}e.
\]
Since \eqref{gwz} is a
basis, $S=f(\Phi)$.

(i)  Let
$\underline{\cal
A}_{\Phi} := (A,
A^*)\in
\ind_0(\underline{{P}})$,
in which $\Phi$ is a
nonsingular matrix
over $\mathbb F$ that
is indecomposable for
similarity and
$A:=\sqrt[\displaystyle
*]{\Phi}$. Let
$g=[G_1, G_2]\in
\End(\underline{\cal
A}_{\Phi})$. Then
\eqref{msp} ensures
that
\begin{equation}\label{mdtc}
G_2A=A G_1,\qquad
G_2A^*=A^* G_1,
\end{equation}
and so
\begin{equation}\label{dkr}
\Phi G_1=A^{-*} AG_1=
A^{-*}G_2 A= G_1A^{-*}
A= G_1{\Phi}.
\end{equation}
Since $G_1$ commutes
with $\Phi$, we have
$G_1 = f(\Phi)$ for
some $f(x)\in \mathbb
F[x]$, and
\begin{equation}\label{lrsh}
G_2=A G_1 A^{-1} = f(A
\Phi A^{-1})= f(A
A^{-*} A A^{-1})
=f(\Phi^{-*}).
\end{equation}

Consequently, the ring
$\End(\underline{\cal
A}_{\Phi})$ of
endomorphisms of
$\underline{\cal
A}_{\Phi}$ consists of
the matrix pairs
\eqref{ldy}, and the
involution
\eqref{kdtc} has the
form
\[
[f(\Phi),f(\Phi^{-*})]^{\circ}=
[f(\Phi^{-*})^*,f(\Phi)^*]=
[\bar
f(\Phi^{-1}),\bar
f(\Phi^*)].
\]

(ii) The first
equality in
\eqref{lrsh} ensures
that each endomorphism
$[f(\Phi),f(\Phi^{-*})]$
is completely
determined by
$f(\Phi)$. Thus, the
ring
$\End(\underline{\cal
A}_{\Phi})$ can be
identified with \[
\mathbb
F[\Phi]=\{f(\Phi)\,|\,f\in\mathbb
F[x]\}\quad \text{with
involution
$f(\Phi)\mapsto \bar
f(\Phi^{-1})$,}\]
which is isomorphic to
$\mathbb
F[x]/p_{\Phi}(x)^s\mathbb
F[x]$, in which
$p_{\Phi}(x)^s$ is the
characteristic
polynomial
\eqref{ser24} of
$\Phi$. Thus, the
radical of the ring $
\mathbb F[\Phi]$ is
generated by
$p_{\Phi}(\Phi)$ and
$\mathbb T({\cal
A}_{\Phi})$ can be
identified with the
field \eqref{ksyq}
with involution
$f(\kappa)^{\circ}=
\bar f(\kappa^{-1})$.

According to Lemma
\ref{LEMMA 7}, each
element of the field
\eqref{ksyq} on which
the involution acts
identically is
uniquely representable
in the form
$q(\kappa)$ for some
nonzero function
$q(x)$ of the form
\eqref{ser13}. The
pair $[q(\Phi), A
q(\Phi) A^{-1}]$ is an
endomorphism of
$\underline{\cal
A}_{\Phi}$ due to
\eqref{mdtc}. This
endomorphism is
selfadjoint since the
function \eqref{ser13}
satisfies
$q(x^{-1})=\bar q(x)$,
and so \[ A q(\Phi)
A^{-1}
=q(\Phi^{-*})=\bar
q(\Phi^*)=
q(\Phi)^*.\]

Since distinct
functions $q(x)$ give
distinct $q(\kappa)$
and
\[
q(\Phi)\in
q(\kappa)=q(\Phi)
+p_{\Phi}(\Phi){\mathbb
F}[\Phi],
\]
in \eqref{lfw} we may
take
$
f_{q(\kappa)}:=
[q(\Phi),q(\Phi)^*]
\in
\End(\underline{\cal
A}_{\Phi}).
$
By \eqref{5mi}, the
corresponding
representations $
{\cal
A}_{\Phi}^{q(\kappa)}
= {\cal
A}_{\Phi}^{f_{q(\kappa)}}
$ have the form
\eqref{wmp} and
constitute the orbit
of ${\cal A}_{\Phi}$.
\end{proof}

\begin{proof}[Proof of
Theorem \ref{Theorem
5}] (a) Each square
matrix $A$ gives the
representation
\eqref{ren7a+} of the
pograph \eqref{jsos+}.
Theorem \ref{tetete1}
ensures that each
representation of
\eqref{jsos+} over a
field $\mathbb F$ of
characteristic
different from $2$ is
isomorphic to a direct
sum of representations
of the form ${\cal
M}^+$ and ${\cal
A}^a$, where ${\cal
M}\in
\ind_1(\underline{P})$,
$\underline{\cal A}\in
\ind_0(\underline{P})$,
and $0\ne
a=a^{\circ}\in\mathbb
T({\cal A})$. This
direct sum is
determined uniquely up
to permutation of
summands and
replacement of the
whole group of
summands $ {\cal
A}^{a_1}
\oplus\dots\oplus
{\cal A}^{a_s} $ with
the same $\cal A$ by $
{\cal A}^{b_1}
\oplus\dots\oplus
{\cal A}^{b_s} $,
provided that the
Hermitian forms $
a_1x_1^{\circ}x_1+\dots+
a_sx_s^{\circ}x_s$ and
$
b_1x_1^{\circ}x_1+\dots+
b_sx_s^{\circ}x_s$ are
equivalent over
$\mathbb T({\cal A})$,
which is a field by
\eqref{ksyq}.

This proves (a) since
we can use the sets
$\ind_1(\underline{P})$
and
$\ind_0(\underline{P})$
from Lemma
\ref{lenhi}; the field
$\mathbb T({\cal A})$
is isomorphic to
\eqref{ksyq}, and the
representations ${\cal
M}^+$ and ${\cal A}^a$
have the form
\eqref{dwh},
\eqref{dld}, and
\eqref{wmp}.

(b) Let $\mathbb F$ be
a real closed field
and let $\Phi\in{\cal
O}_{\mathbb F}$ be
such that
$\sqrt[\displaystyle
*]{\Phi}$ exists. Let
us identify $\mathbb
T({\cal A}_{\Phi})$
with the field
\eqref{ksyq}. Then
$\mathbb T({\cal
A}_{\Phi})$ is either
$\mathbb F$ or its
algebraic closure. In
the latter case, the
involution \eqref{kxu}
on $\mathbb T({\cal
A}_{\Phi})$ is not the
identity; otherwise
$\kappa= \kappa^{-1}$,
$\kappa^2-1=0$, i.e.,
$p_{\Phi}(x)=x^2-1$,
which contradicts the
irreducibility of
$p_{\Phi}(x)$.

Applying Theorem
\ref{tetete}, we
complete the proof of
(b).
\end{proof}

\section{Proof of Theorem
\ref{t1.1}}
\label{secmat}

\subsection{Proof of
Theorem \ref{t1.1}(a)}
\label{sec(a)}

Let $\mathbb F$ be an
algebraically closed
field of
characteristic
different from $2$
with the identity
involution. Take
${\cal O}_{\mathbb F}$
to be all nonsingular
Jordan blocks.

The summands
(i)--(iii) of Theorem
\ref{t1.1}(a) can be
obtained from the
summands (i)--(iii) of
Theorem \ref{Theorem
5} because for nonzero
$\lambda,\mu\in\mathbb
F$
\begin{align*}
J_n(\lambda)\text{ is
similar to
}J_n(\mu)^{-T}&
\quad\Longleftrightarrow\quad
\lambda={\mu}^{-1},\\
\sqrt[
T]{J_n(\lambda)}\
\text{ exists
}&\quad\Longleftrightarrow\quad
\lambda=(-1)^{n+1}.
\end{align*}
The first of these two
equivalences is
obvious.

Let us prove the
second. By
\eqref{lbdr} and
\eqref{4.adlw}, if
$\sqrt[T]{J_n(\lambda)}$
exists then
$\lambda=(-1)^{n+1}$.
Conversely, let
$\lambda=(-1)^{n+1}$.
It suffices to prove
the following useful
statement:
\begin{equation}\label{cnt}
\text{the cosquares of
$\Gamma_n$ and
$\Gamma_n'$ are
similar to
$J_n((-1)^{n+1})$},
\end{equation}
which implies that
$\sqrt[T]{J_n((-1)^{n+1})}$
exists by \eqref{ndw}
with
$\sqrt[T]{\Phi}=\Gamma_n$
and
$\Psi=J_n((-1)^{n+1})$.

To verify the first
similarity in
\eqref{cnt}, compute
\begin{equation*}\label{1n}
\Gamma_n^{-1}=(-1)^{n+1}
 \begin{bmatrix}
\vdots&\vdots&\vdots&\vdots&\ddd
\\
-1&-1&-1&-1&\\ 1&1&1&&\\ -1&-1&&&\\
1&&&& 0
\end{bmatrix}
\end{equation*}
and
\begin{equation}\label{1x11}
\Gamma_n^{-T}\Gamma_n=
(-1)^{n+1}
\begin{bmatrix} 1&2&&\text{\raisebox{-6pt}
{\large\rm *}}
\\&1&\ddots&\\
&&\ddots&2\\
0 &&&1
\end{bmatrix}.
\end{equation}

To verify the second
similarity in
\eqref{cnt}, there are
two cases to consider:
If $n$ is even then
\[
(\Gamma'_{n})^{-1}=
\left[\begin{array}{c|c}
\begin{matrix}
 \vdots&\vdots&&\vdots
    \\
1&1&\cdots&1\\
 -1&-1&\cdots&-1
    \\
 1&1&\cdots&1
\end{matrix}
&
\begin{matrix}
\vdots&\ddd&-1&1
    \\
1&\ddd&\ddd
    \\
 -1&1&&
    \\
 1&&&
\end{matrix}
    \\ \hline
\begin{matrix}
 -1&-1&\cdots&-1
 \\
\vdots&\vdots&\ddd&
 \\
-1&-1&&
 \\
-1&&&
\end{matrix}
 &\text{\large\rm 0}
\end{array}\right]
\]
and
\begin{equation*}\label{kiy}
(\Gamma'_n)^{-T}
\Gamma'_n=
\begin{bmatrix}
-1&\pm 2&
&\text{\raisebox{-6pt}
{\large\rm *}}
 \\
&-1&\ddots&
  \\
&&\ddots&\pm 2
   \\
0&&& -1
\end{bmatrix}.
\end{equation*}
If $n$ is odd then
\begin{equation}\label{1n'}
(\Gamma'_n)^{-1}=
\begin{bmatrix}
&&&\pm 1&\dots&-1&1\\
&0&&\vdots&\ddd& \ddd\\
&&&-1&1\\
&&&1\\
  &&1\\
  &\ddd\\
  1&&&&&&0
\end{bmatrix}
\end{equation}
and
\begin{equation}\label{1n''}
(\Gamma'_n)^{-T}
\Gamma'_n=
\begin{bmatrix}
1&\pm 1&
&\text{\raisebox{-6pt}
{\large\rm *}}
 \\
&1&\ddots&
  \\
&&\ddots&\pm 1
   \\
0&&& 1
\end{bmatrix}.
\end{equation}

We have proved that
all direct sums of
matrices of the form
(i)--(iii) are
canonical matrices for
congruence. Let us
prove the last
statement of Theorem
\ref{t1.1}(a). If two
nonsingular matrices
over $\mathbb F$ are
congruent then their
cosquares are similar.
The converse statement
is correct too because
the cosquares of
distinct canonical
matrices for
congruence have
distinct Jordan
canonical forms. Due
to \eqref{cnt},
$\Gamma_{n}$ and
$\Gamma'_{n}$ are
congruent to
$\sqrt[T]{J_n((-1)^{n+1})}$.

\subsection{Proof of
Theorem \ref{t1.1}(b)}
\label{sec(b)}

Let $\mathbb F$ be an
algebraically closed
field of
characteristic $2$.

According to
\cite{ser0}, each
square matrix over\/
$\mathbb F$ is
congruent to a matrix
of the form
\begin{equation}\label{ogy}
\bigoplus_i
[J_{m_i}(\lambda_i)\diag
I_{m_i}]\oplus
\bigoplus_j
\sqrt[T]{J_{n_j}(1)}
\oplus \bigoplus_k
J_{r_k}(0),
\end{equation}
in which $\lambda_i\ne
0$, $n_j$ is odd, and
$J_{m_i}(\lambda_i)\ne
J_{n_j}(1)$ for all
$i$ and $j$. This
direct sum is
determined uniquely up
to permutation of
summands and
replacement of any
$J_{m_i}(\lambda_i)$
by
$J_{m_i}(\lambda_i^{-1})$.

The matrix $\sqrt[
T]{J_{n}(1)}$ was
constructed in
\cite[Lemma 1]{ser0}
for any odd $n$, but
it is cumbersome. Let
us prove that
$\Gamma'_{n}$ is
congruent to $\sqrt[
T]{J_{n}(1)}$. Due to
\eqref{1n'} and
\eqref{1n''} (with
$-1=1$), the cosquare
of $\Gamma'_{n}$ is
similar to $J_{n}(1)$.
Let $\Sigma$ be the
canonical matrix of
the form \eqref{ogy}
for $\Gamma'_{n}$.
Then the cosquares of
$\Sigma$ and
$\Gamma'_{n}$ are
similar, and so
$\Sigma=\sqrt[
T]{J_{n}(1)}$.

\subsection{Proof of
Theorem \ref{t1.1}(c)}
\label{sec(c)}

Let $\mathbb F=\mathbb
P+\mathbb Pi$ be an
algebraically closed
field with nonidentity
involution represented
in the form
\eqref{1pp11}. Take
${\cal O}_{\mathbb F}$
to be all nonsingular
Jordan blocks.

The summands
(i)--(iii) of Theorem
\ref{t1.1}(c) can be
obtained from the
summands (i)--(iii) of
Theorem \ref{Theorem
5} because for nonzero
$\lambda,\mu\in\mathbb
F$
\begin{align}\nonumber
J_n(\lambda)\text{ is
similar to
}J_n(\mu)^{-*}
&\quad\Longleftrightarrow\quad
\lambda=\bar{\mu}^{-1},\\
\label{nlsi}
\sqrt[\displaystyle
*]{J_n(\lambda)}\
\text{ exists
}&\quad\Longleftrightarrow\quad
|\lambda|=1\ \
(\text{see
\eqref{1kk}}).
\end{align}
Let us prove
\eqref{nlsi}. By
\eqref{lbdr}, if
$\sqrt[\displaystyle
*]{J_n(\lambda)}$
exists for $\lambda
=a+bi\ (a,b\in \mathbb
P)$ then $x-\lambda
=x-\bar\lambda^{-1}$.
Thus, $\lambda
=\bar\lambda^{-1}$ and
$1=\lambda
\bar\lambda=a^2+b^2=|\lambda
|^2$. Conversely, let
$|\lambda|=1$. It
suffices to show that
the *cosquare of
$i^{n+1}\sqrt{\lambda}
\Gamma_n$ is similar
to $J_n(\lambda)$
since then
$\sqrt[\displaystyle
*]{J_n(\lambda)}$
exists by \eqref{ndw}
with
$\Psi=J_n(\lambda)$.
To verify this
similarity, observe
that for each
unimodular
$\lambda\in\mathbb F$,
\begin{equation}\label{kum}
(i^{n+1}\sqrt{\lambda}
\Gamma_n)^{-*}
(i^{n+1}\sqrt{\lambda}
\Gamma_n)=
\lambda\,(-1)^{n+1}
\Gamma_n^{-T}\Gamma_n;
\end{equation}
by \eqref{1x11},
$\lambda\,(-1)^{n+1}
\Gamma_n^{-T}\Gamma_n$
is similar to $\lambda
J_n (1)$, which is
similar to $J_n
(\lambda).$

It remains to prove
that each of the
matrices \eqref{jde}
can be used instead of
(iii) in Theorem
\ref{t1.1}(c). Let us
show that if
$\lambda\in\mathbb F$
is unimodular, then
$J_n(\lambda)$ is
similar to the
*cosquare of each of
the matrices
\begin{equation}\label{ksim}
\sqrt{\lambda}\sqrt[\displaystyle
*]{ J_n (1)},\qquad
i^{n+1}\sqrt{\lambda}
\Gamma_n,\qquad
i^{n+1}\sqrt{\lambda}
\Gamma'_n,\qquad
\sqrt{\lambda}\,
\Delta_n(1).
\end{equation}
The first similarity
is obvious. The second
was proved in
\eqref{kum}. The third
can be proved
analogously since
$(\Gamma_n')^{-T}\Gamma'_n$
is similar to
$\Gamma_n^{-T}\Gamma_n$
by \eqref{cnt}. The
fourth similarity
holds since $J_n(1)$
is similar to the
*cosquare of
$\Delta_n(1)$ as a
consequence of the
following useful
property: for each
$\mu\in \mathbb F$
with $\bar\mu^{-1}
\mu\ne -1$,
\begin{equation}\label{kdt}
\text{
$J_n(\bar{\mu}^{-1}
\mu)$ is similar to
the *cosquare of
$\Delta_n(\mu)$.}
\end{equation}
To verify this
assertion, compute
\begin{multline*}
\Delta_n(\mu)^{-*}
\Delta_n(\mu)\\=
 \begin{bmatrix}
\text{\raisebox{-6pt}
{\large\rm
*}}&&i\bar{\mu}^{-2}&\
\bar{\mu}^{-1}\\
&\ddd&\ddd&\\
i\bar{\mu}^{-2}&\
\bar{\mu}^{-1}&&\\
\bar{\mu}^{-1}&&&0
\end{bmatrix}
\Delta_n(\mu)=
\begin{bmatrix}
\mu\bar{\mu}^{-1}
&i\bar{\mu}^{-1}u&&
\text{\raisebox{-6pt}
{\large\rm *}}
\\&\mu\bar{\mu}^{-1}
&\ddots&\\
&&\ddots&
i\bar{\mu}^{-1}u\\
0 &&&\mu\bar{\mu}^{-1}
\end{bmatrix}
\end{multline*}
with $
u:=\bar{\mu}^{-1} \mu+
1\ne 0.$

Therefore, the
*cosquare of each of
the matrices
\eqref{ksim} can
replace $J_n(\lambda)$
in ${\cal O}_{\mathbb
F}$, and so each of
the matrices
\eqref{ksim} may be
used as
$\sqrt[\displaystyle
*]{\Phi}$ in (iii) of
Theorem \ref{Theorem
5}(a). Thus, instead
of
$\pm\sqrt[\displaystyle
*]{J_n(\lambda)}$ in
(iii) of Theorem
\ref{t1.1}(c) we may
use any of the
matrices \eqref{ksim}
multiplied by $\pm 1$;
and hence any of the
matrices \eqref{jde}
except for $\mu A$
since each
$\sqrt{\lambda}$ can
be represented in the
form $a+bi$ with
$a,b\in\mathbb P$,
$b\ge 0$, and $a+bi\ne
-1$. Let $A$ be any
nonsingular $n\times
n$ matrix whose
*cosquare is similar
to a Jordan block.
Then $A$ is *congruent
to some matrix of type
(iii), and hence $A$
is *congruent to
$\mu_0\Gamma_n$ for
some unimodular
$\mu_0$. Thus, $\mu A$
is *congruent to
$\mu\mu_0\Gamma_n$,
and so we may use $\mu
A$ instead of
$\pm\sqrt[\displaystyle
*]{J_n(\lambda)}$ in
(iii).

\subsection{Proof of
Theorem \ref{t1.1}(d)}
\label{sec(d)}

Let $\mathbb P$ be a
real closed field. Let
$\mathbb K:=\mathbb
P+\mathbb Pi$  be the
algebraic closure of
$\mathbb P$
represented in the
form \eqref{1pp} with
involution
$a+bi\mapsto a-bi$. By
\cite[Theorem
3.4.5]{hor}, we may
take ${\cal
O}_{\mathbb P}$ to be
all $J_n(a)$ with $a
\in\mathbb P$, and all
$J_n(\lambda
)^{\mathbb P}$ with
$\lambda\in{\mathbb
K}\smallsetminus\mathbb
P$ determined up to
replacement by
$\bar\lambda$.

Let $a \in\mathbb P$.
Reasoning as in the
proof of Theorem
\ref{t1.1}(a), we
conclude that
\begin{itemize}
  \item
$J_n(a)$ is similar to
$J_n(b)^{-T}$ with $b
\in\mathbb P$ if and
only if $a=b^{-1}$;

  \item
$\sqrt[ T]{J_n(a)}$
exists if and only if
$a=(-1)^{n+1}$.
\end{itemize}
Thus, the summands
(i)--(iii) of Theorem
\ref{Theorem 5} give
the summands
(i)--(iii) in Theorem
\ref{t1.1}(d). Due to
\eqref{cnt}, we may
take $(\Gamma_n)^{-T}
\Gamma_n$ or
$(\Gamma'_n)^{-T}
\Gamma'_n$ instead of
$J_n((-1)^{n+1})$ in
${\cal O}_{\mathbb
P}$. Thus, we may use
$\pm \Gamma_n$ or $\pm
\Gamma_n'$ instead of
$\pm \sqrt[T]{
J_n((-1)^{n+1})}$ in
Theorem \ref{t1.1}(d).

Let
$\lambda,\mu\in(\mathbb
P+\mathbb
Pi)\smallsetminus\mathbb
P$. Then
\begin{align}\nonumber
J_n(\lambda)^{\mathbb
P}\text{ is similar to
}(J_n(\mu)^{\mathbb
P})^{-T}
&\quad\Longleftrightarrow\quad
\lambda\in\{{\mu}^{-1},
\bar{\mu}^{-1}\},\\
\label{ndsis}
\sqrt[T]{J_n(\lambda)^{\mathbb
P}}\ \text{ exists
}&\quad\Longleftrightarrow\quad
|\lambda |=1.
\end{align}
Let us prove
\eqref{ndsis}. For
$\Phi:=J_n(\lambda
)^{\mathbb P}$, we
have
\begin{equation*}\label{kgv}
p_{\Phi}(x)=(x-\lambda)(x-\bar\lambda)
=x^2-(\lambda+\bar\lambda)
+|\lambda|^2.
\end{equation*}
If $\sqrt[ T]{\Phi}$
exists then
$|\lambda|=1$ by
\eqref{lbdr} and
\eqref{lyf1}.

Conversely, let
$|\lambda|=1$. We can
take
\begin{equation}\label{mdu}
\sqrt[T]{J_n(\lambda)^{
\mathbb
P}}=\Big(\sqrt[\displaystyle
*]{J_n(\lambda)}\Big)^{
\mathbb P}.
\end{equation}
Indeed,
$M:=\sqrt[\displaystyle
*]{J_n(\lambda )}$
exists by
\eqref{nlsi}; it
suffices to prove
\begin{equation}\label{hfe}
(M^{\mathbb P})^{-T}
M^{\mathbb
P}=J_n(\lambda
)^{\mathbb P}.
\end{equation}
If $M$ is represented
in the form $M=A+Bi$
with $A$ and $B$ over
$\mathbb P$, then its
realification
$M^{\mathbb P}$ (see
\eqref{1j}) is
permutationally
similar to
\[
M_{\mathbb
P}:=\begin{bmatrix}
  A&-B\\B&A
\end{bmatrix}.
\]
Applying the same
transformation of
permutation similarity
to the matrices of
\eqref{hfe} gives
\begin{equation}\label{jpoi}
(M_{\mathbb P})^{-T}
M_{\mathbb
P}=J_n(\lambda
)_{\mathbb P}.
\end{equation}
Since
\[
\begin{bmatrix}
  A+Bi&0\\0&A-Bi
\end{bmatrix}
\begin{bmatrix}
  I&iI\\I&-iI
\end{bmatrix}=
\begin{bmatrix}
  I&iI\\I&-iI
\end{bmatrix}
\begin{bmatrix}
  A&-B\\B&A
\end{bmatrix},
\]
we have
\begin{equation*}\label{luf}
M_{\mathbb
P}=S^{-1}(M\oplus \bar
M)S=S^{*}(M\oplus \bar
M)S
\end{equation*}
with
\[
S:=\frac{1}{\sqrt{2}}
\begin{bmatrix}
  I&iI\\I&-iI
\end{bmatrix}=S^{-*}.
\]
Thus, \eqref{jpoi} is
represented in the
form
\[
\left(S^*(M\oplus \bar
M)S\right)^{-*}
S^*(M\oplus \bar M)S=
S^{-1}\left(J_n(\lambda
)\oplus
J_n(\bar\lambda
)\right)S.
\]
This equality is
equivalent to the pair
of equalities
\[
M^{-*} M=J_n(\lambda
),\qquad \bar M^{-*}
\bar M=J_n(\bar
\lambda ),
\]
which are  valid since
$M=\sqrt[\displaystyle
*]{J_n(\lambda )}$.
This proves
\eqref{mdu}, which
completes the proof of
\eqref{ndsis}.

Thus, the summands
(ii) and (iii) of
Theorem \ref{Theorem
5} give the summands
(ii$'$) and (iii$'$)
in Theorem
\ref{t1.1}(d).

It remains to prove
that each of the
matrices \eqref{dsk}
can be used instead of
(iii$'$). Every
unimodular
$\lambda=a+bi\in
{\mathbb P}+{\mathbb
P}i$ with $b>0$ can be
expressed in the form
\begin{equation}\label{roo4}
\lambda=\frac{e+i}{e-i}\,,
\qquad e\in{\mathbb
P},\quad e>0.
\end{equation}
Due to \eqref{cnt},
the *cosquares
\[
((e+i)\Gamma_n)^{-*}
(e+i)\Gamma_n=\lambda
\Gamma_n^{-*}\Gamma_n,\quad
((e+i)\Gamma'_n)^{-*}
(e+i)\Gamma'_n=\lambda
(\Gamma'_n)^{-*}\Gamma'_n
\]
are similar to
$\lambda
J_n((-1)^{n+1})$,
which is similar to
$(-1)^{n+1}
J_n(\lambda)$. Theorem
\ref{Theorem 5}
ensures that the
matrix $\pm\sqrt[T]{
J_n(\lambda)^{\,\mathbb
P}}$ in (iii$'$) can
be replaced
\begin{equation}\label{jyf}
\text{by
$\pm((e+i)\Gamma_n)^{\mathbb
P}$ and also by
$\pm((e+i)\Gamma'_n)^{\mathbb
P}$ with $e>0$.}
\end{equation}
For each square matrix
$A$ over $\mathbb
P+\mathbb Pi$ we have
\begin{equation}\label{roo343}
S^TA^{\mathbb
P}S=\overline
A^{\:\mathbb P},\qquad
S:=\diagg
(1,-1,1,-1,\dots),
\end{equation}
and so
$-\big((e+i)\Gamma_n\big)^{\mathbb
P}$ is congruent to
\[
 -\overline{
(e+i)\Gamma_n}^{\,\mathbb
P}=
-\big((e-i)\Gamma_n\big)^{\mathbb
P}=
\big((-e+i)\Gamma_n\big)^{\mathbb
P}.
\]
Therefore, the
matrices \eqref{jyf}
are congruent to
$((c+i)\Gamma_n)^{\mathbb
P}$ and
$((c+i)\Gamma'_n)^{\mathbb
P}$ with $0\ne
c\in{\mathbb P}$ and
$|c|=e$.

Let us show that the
summands {\rm(iii$'$)}
can be also replaced
by $\Delta_n(c+i)$
with $0\ne
c\in{\mathbb P}$. By
\eqref{kdt}, the
*cosquare of
$\Delta_n(e+i)$ with
$e>0$ is similar to
$J_n(\lambda)$, in
which $\lambda$ is
defined by
\eqref{roo4}.
Reasoning as in the
proof of \eqref{hfe},
we find that the
cosquare of
$\Delta_n(e+i)^{\mathbb
P}$ is similar to
$J_n(\lambda)^{\,\mathbb
P}$. Hence,
$\pm\Delta_n(e+i)^{\mathbb
P}$ with $e>0$ can be
used instead of
(iii$'$). Due to
\eqref{roo343}, the
matrix
$-\Delta_n(e+i)^{\mathbb
P}$ is congruent to
\[
 -\overline{
\Delta_n(e+i)}^{\mathbb
P}=
\Delta_n(-e+i)^{\mathbb
P}.
\]

\subsection{Proof of
Theorem \ref{t1.1}(e)}
\label{sec(e)}

\begin{lemma}
 \label{le}
Let $\mathbb H$ be the
skew field of
quaternions over a
real closed field
$\mathbb P$. Let
${\cal O}_{\mathbb H}$
be a maximal set of
nonsingular
indecomposable
canonical matrices
over $\mathbb H$ for
similarity.

\begin{itemize}
  \item[\rm(a)]
Each square matrix
over $\mathbb H$ is
*congruent to a direct
sum, determined
uniquely up to
permutation of
summands, of matrices
of the form:
\begin{itemize}
  \item[\rm(i)] $J_n(0)$.

  \item[\rm(ii)]
$(\Phi,
I_n)^+=[\Phi\diag
I_n]$, in which
$\Phi\in{\cal
O}_{\mathbb H}$ is an
$n\times n$ matrix
such that
$\sqrt[\displaystyle
*]{\Phi}$ does not
exist; $\Phi$ is
determined up to
replacement by the
unique $\Psi\in{\cal
O}_{\mathbb F}$ that
is similar to
$\Phi^{-*}$.

  \item[\rm(iii)]
$\varepsilon_{\Phi}
\sqrt[\displaystyle
*]{\Phi}$, in which
$\Phi\in{\cal
O}_{\mathbb H}$ is
such that
$\sqrt[\displaystyle
*]{\Phi}$ exists;
$\varepsilon_{\Phi}
=1$ if
$\sqrt[\displaystyle
*]{\Phi}$ is
*congruent to
$-\sqrt[\displaystyle
*]{\Phi}$ and
$\varepsilon_{\Phi}
=\pm 1$ otherwise.
This means that
$\varepsilon_{\Phi}
=1$ if and only if
$\mathbb T({\cal
A}_{\Phi})$ is an
algebraically closed
field with the
identity involution or
$\mathbb T({\cal
A}_{\Phi})$ is a skew
field of quaternions
with involution
different from
quaternionic
conjugation
\eqref{ne}.
\end{itemize}
\end{itemize}

\item[\rm(b)] If
$\varepsilon_{\Phi} =
1$ and $\Phi$ is
similar to $\Psi$,
then
$\varepsilon_{\Psi}
=1$.
\end{lemma}

\begin{proof}
(a) Theorem
\ref{tetete} ensures
that any given
representation of any
pograph $P$ over
$\mathbb H$ decomposes
uniquely, up to
isomorphism of
summands, into a
direct sum of
indecomposable
representations. Hence
the problem of
classifying
representations of $P$
reduces to the problem
of classifying
indecomposable
representations. By
Theorem \ref{tetete}
and Lemma \ref{lenhi},
the matrices
(i)--(iii) form a
maximal set of
nonisomorphic
indecomposable
representations of the
pograph \eqref{jsos+}.

(b) On the contrary,
assume that
$\varepsilon_{\Psi}
=\pm 1$. Then
$\sqrt[\displaystyle
*]{\Psi}$ and
$-\sqrt[\displaystyle
*]{\Psi}$ have the
same canonical form
$\sqrt[\displaystyle
*]{\Phi}$, a
contradiction.
\end{proof}

Let $\mathbb P$ be a
real closed field and
let $\mathbb H$ be the
skew field of $\mathbb
P$-quaternions with
quaternionic
conjugation \eqref{ne}
or quaternionic
semiconjugation
\eqref{nen}. These
involutions act as
complex conjugation on
the algebraically
closed subfield
$\mathbb K:=\mathbb
P+\mathbb Pi$. By
\cite[Section 3, \S
12]{jac}, we can take
${\cal O}_{\mathbb F}$
to be all
$J_n(\lambda)$, in
which
$\lambda\in\mathbb K$
and $\lambda$ is
determined up to
replacement by $\bar
\lambda$. For any
nonzero $\mu\in\mathbb
K$, the matrix
$J_n(\mu)^{-*}$ is
similar to
$J_n(\bar\mu^{-1})$.
Since $\bar\mu^{-1}$
is determined up to
replacement by
$\mu^{-1}$,
\begin{equation*}\label{lust}
J_n(\lambda)\text{ is
similar to
}J_n(\mu)^{-*}
\quad\Longleftrightarrow\quad
\lambda\in\{{\mu}^{-1},
\bar{\mu}^{-1}\}.
\end{equation*}

Let us prove that for
a nonzero $\lambda
\in\mathbb K$
\begin{equation*}\label{nsis}
\sqrt[\displaystyle
*]{J_n(\lambda)}\
\text{ exists
}\quad\Longleftrightarrow\quad
|\lambda |=1.
\end{equation*}
If
$\sqrt[\displaystyle
*]{J_n(\lambda)}$
exists then by
\eqref{lbdr}
$x-\lambda
=x-\bar\lambda^{-1}$
and so $|\lambda|=1$.
Conversely, let
$|\lambda|=1$. In view
of \eqref{1x11}, the
*cosquare of
$A:=\sqrt{\lambda
(-1)^{n+1}}\Gamma_n$
is
\begin{equation}\label{ntd}
\Phi:=A^{-*} A=
\lambda F,\qquad
F:=(-1)^{n+1}
\Gamma_n^{-T}\Gamma_n
=
\begin{bmatrix}
1&2&&\text{\raisebox{-6pt}
{\large\rm *}}
\\&1&\ddots&\\
&&\ddots&2\\
0 &&&1
\end{bmatrix},
\end{equation}
and so $\Phi$ is
similar to
$J_n(\lambda)$. Thus,
$\sqrt[\displaystyle
*]{J_n(\lambda)}$
exists by \eqref{ndw}
with
$\sqrt[\displaystyle
*]{\Phi}=A$.

Lemma \ref{le}(a)
ensures the summands
(i)--(iii) in Theorem
\ref{t1.1}(e); the
coefficient
$\varepsilon$ in (iii)
is defined in Lemma
\ref{le}(a). Let us
prove that
$\varepsilon$ can be
calculated by
\eqref{kki}. By Lemma
\ref{le}(b) and since
$\Phi$ in \eqref{ntd}
is similar to
$J_n(\lambda)$, we
have $\varepsilon
=\varepsilon_{\Phi}$,
so it suffices to
prove \eqref{kki} for
$\varepsilon_{\Phi}$.

Two matrices
$G_1,G_2\in\mathbb
H^{n\times n}$ give an
endomorphism
$[G_1,G_2]$ of
$\underline{\cal
A}_{\Phi}=(A,A^*)$ if
and only if they
satisfy \eqref{mdtc}.
By \eqref{dkr}, the
equalities
\eqref{mdtc} imply
\begin{equation}\label{kdn}
G_1{\Phi}={\Phi} G_1.
\end{equation}

\emph{Case $\lambda\ne
\pm 1$}. Represent
$G_1$ in the form
$U+Vj$ with
$U,V\in\mathbb
K^{n\times n}$. Then
\eqref{kdn} implies
two equalities
\begin{equation}\label{kuf}
U\Phi =\Phi U,\qquad
V\bar \Phi j =\Phi Vj.
\end{equation}
By the second equality
and \eqref{ntd},
$\bar\lambda
VF=\lambda FV$,
\[
(\bar\lambda-\lambda)V=
\lambda
(F-I)V-\bar\lambda
V(F-I).
\]
Thus $V=0$ since
$\lambda\ne\bar\lambda$
and $F-I$ is nilpotent
upper triangular. By
the first equality in
\eqref{kuf} (which is
over the field
$\mathbb K$),
$G_1=U=f(\lambda
F)=f(\Phi)$ for some
$f\in\mathbb K[x]$;
see the beginning of
the proof of Lemma
\ref{lenhi}(d). Since
$A$ is over $\mathbb
K$, the identities
\eqref{mdtc} imply
\eqref{lrsh}.

Because $G_2=A G_1
A^{-1}$, the
homomorphism
$[G_1,G_2]\in \End
(\underline{\cal
A}_{\Psi})$ is
completely determined
by $G_1=f(\Phi)$. The
matrix $\Phi=\lambda
F$ is upper
triangular, so the
mapping
$f(\Phi)\mapsto
f(\lambda)$ on
$\mathbb K[\Phi]$
defines an
endomorphism of rings
$\End (\underline{\cal
A}_{\Phi})\to \mathbb
K$; its kernel is the
radical of $\End
(\underline{\cal
A}_{\Phi})$. Hence
$\mathbb T( {\cal
A}_{\Phi})$ can be
identified with
$\mathbb K$. Using
\eqref{ldbye}, we see
that the involution on
$\mathbb T({\cal
A}_{\Phi})$ is induced
by the mapping
$G_1\mapsto G_2^*$ of
the form
\[f(\lambda
F)\mapsto f((\lambda
F)^{-*})^*= \bar
f((\lambda F)^{-1}).\]
Therefore, the
involution is
\[f(\lambda)\ \longmapsto\ \bar
f({\lambda}^{-1})=
\bar f(\bar{\lambda})=
\overline{f({\lambda})}\]
and coincides with the
involution
$a+bi\mapsto a-bi$ on
$\mathbb K$. The
statement (iii) in
Lemma \ref{le}(a) now
implies
$\varepsilon_{\Phi}=
\pm 1$; this proves
\eqref{kki} in the
case $\lambda\ne \pm
1$.
\medskip

\emph{Case $\lambda=
\pm 1$.} Then
\begin{equation}\label{mse1}
A=\sqrt{\lambda
(-1)^{n+1}}\Gamma_n=
  \begin{cases}
    \Gamma_n & \text{if $
    \lambda=(-1)^{n+1}$}, \\
    i\Gamma_n & \text{if $
    \lambda=(-1)^{n}$}.
  \end{cases}
\end{equation}

Define
\begin{align*}
\check
h&:=a+bi-cj-dk\quad
\text{for each}\ \
h=a+bi+cj+dk\in\mathbb
H,
\\
\check f(x)&:=\sum_l
\check h_lx^l\quad
\text{for each}\ \
f(x)=\sum_l
h_{l}x^l\in \mathbb
H[x].
\end{align*}
Because $\lambda= \pm
1$ and by \eqref{kdn},
$G_1$ has the form
\[
G_1=\begin{bmatrix}
 a_1&a_2&\ddots&a_{n}
 \\&a_1&\ddots&\ddots
 \\&&\ddots&a_2\\
 0&&&a_1
\end{bmatrix}, \qquad
a_1,\dots,a_n\in\mathbb
H.
\]
Thus, $G_1=f(\Phi)$
for some polynomial
$f(x)\in\mathbb H[x]$.

Using the first
equality in
\eqref{mdtc}, the
identity $if(x)=\check
f(ix)$,  and
\eqref{mse1}, we
obtain
\begin{align*}
G_2=A G_1 A^{-1} =A
f(\Phi) A^{-1}
=\begin{cases}
   f( A
\Phi A^{-1})= f(
\Phi^{-*}) & \text{if
$
    \lambda=(-1)^{n+1}$,} \\
    \check f( A
\Phi A^{-1})= \check
f( \Phi^{-*}) &
\text{if $
    \lambda=(-1)^{n}$}.
  \end{cases}
\end{align*}

Since the homomorphism
$[G_1,G_2]$ is
completely determined
by $G_1=f(\Phi)$ and
$\Phi$ has the upper
triangular form
\eqref{ntd} with
$\lambda=\pm 1$, we
conclude that the
mapping
$f(\Phi)\mapsto
f(\lambda)$ defines an
endomorphism of rings
$\End(\underline{\cal
A}_{\Phi})\to \mathbb
H$; its kernel is the
radical of
$\End(\underline{\cal
A}_{\Phi})$. Hence
$\mathbb T( {\cal
A}_{\Phi})$ can be
identified with
$\mathbb H$. The
involution on $\mathbb
T( {\cal A}_{\Phi})$
is induced by the
mapping $G_1\mapsto
G_2^*$; i.e., by
\[
f(\Phi)\mapsto
\begin{cases}
   \bar{f}( \Phi^{-1}) &
\text{if $
    \lambda=(-1)^{n+1}$,} \\
   \widehat{f}(
\Phi^{-1}) & \text{if
$ \lambda=(-1)^{n}$},
  \end{cases}
\]
in which the
involution
$h\mapsto\bar h$ on
$\mathbb F$ is either
quaternionic
conjugation \eqref{ne}
or quaternionic
semiconjugation
\eqref{nen}, and
$h\mapsto\widehat{h}$
denotes the other
involution \eqref{nen}
or \eqref{ne}. Thus
the involution on
$\mathbb T( {\cal
A}_{\Phi})$ is
$h\mapsto\bar h$ if
$\lambda=(-1)^{n+1}$
and is $h\mapsto
\widehat{h}$ if
$\lambda=(-1)^{n}$.
Due to (iii) in Lemma
\ref{le}(a), this
proves \eqref{kki} in
the case $\lambda= \pm
1$.
\medskip

It remains to prove
that the matrices
\eqref{gyo} and
\eqref{gyo1} can be
used instead of (iii)
in Theorem
\ref{t1.1}(e).

Let us prove this
statement for the
first matrix in
\eqref{gyo}. For each
unimodular
$\lambda\in\mathbb K$,
the *cosquare
\eqref{ntd} of
$A=\sqrt{\lambda
(-1)^{n+1}}\Gamma_n$
is similar to
$J_n(\lambda)$, so we
can replace
$J_n(\lambda)$ by
$\Phi$ in ${\cal
O}_{\mathbb H}$ and
conclude by Lemma
\ref{le}(a) that
$\varepsilon A$ can be
used instead of (iii)
in Theorem
\ref{t1.1}(e).

First, let the
involution on $\mathbb
H$ be quaternionic
conjugation. By
\eqref{kki} the matrix
$\varepsilon A$ is
\begin{equation}\label{su1a}
\text{either }\
i\Gamma_n,\ \text{ or
}\ \pm\mu\Gamma_n
\text{ with }
\mu:=\sqrt{\lambda(-1)
^{n+1}}\ne i.
\end{equation}
Since $\lambda$ is
determined up to
replacement by
$\bar{\lambda}$ and
$\sqrt{\lambda(-1)
^{n+1}}\ne i$, we can
take
$\lambda(-1)^{n+1}=u+vi\ne
-1$ with $v\ge 0$, and
obtain
$\mu=\sqrt{\lambda(-1)^{n+1}}=a+bi$
with $a> 0$ and $b\ge
0$. Replacing the
matrices
$-\mu\Gamma_n=(-a-bi)\Gamma_n$
in \eqref{su1a} by the
*congruent matrices
$\bar{j}\cdot(-a-bi)\Gamma_n\cdot
j=(-a+bi)\Gamma_n$, we
get the first matrix
in \eqref{gyo}.

Now let the involution
be quaternionic
semiconjugation. By
\eqref{kki} the matrix
$\varepsilon A$ is
\begin{equation}\label{su2}
\text{either \
$\Gamma_n$, \ or \
$\pm\mu\Gamma_n$ with
}
\mu:=\sqrt{\lambda(-1)
^{n+1}}\ne 1.
\end{equation}
In \eqref{su2} we can
take
$\lambda(-1)^{n+1}=u+vi\ne
1$ with $v\ge 0$. Then
$\mu=\sqrt{\lambda(-1)^{n+1}}=a+bi$
with $a\ge 0$ and $b>
0$. Replacing the
matrices
$-\mu\Gamma_n=(-a-bi)\Gamma_n$
in \eqref{su2} by the
*congruent matrices
$\bar{j}\cdot(-a-bi)\Gamma_n\cdot
j=(a-bi)\Gamma_n$
($\bar j=j$ since the
involution is
quaternionic
semiconjugation), we
get the first matrix
in \eqref{gyo}.

The same reasoning
applies to the second
matrix in \eqref{gyo}.

Let us prove that the
matrix \eqref{gyo1}
can be used instead of
(iii) in Theorem
\ref{t1.1}(e). By
\eqref{kdt},
$J_n(\lambda)$ with a
unimodular
$\lambda\in\mathbb K$
is similar to the
*cosquare of
$\sqrt{\lambda}\,
\Delta_n$ with
$\Delta_n:=\Delta_n(1)$.
Therefore,
$\varepsilon
\sqrt[\displaystyle
*]{ J_n (\lambda)}$ in
(iii) can be replaced
by $\varepsilon
\sqrt{\lambda}\,
\Delta_n$.

Suppose that either
the involution is
quaternionic
conjugation and $n$ is
odd, or that the
involution is
quaternionic
semiconjugation and
$n$ is even. Then
$\bar j=(-1)^{n}j$. By
\eqref{kki},
$\varepsilon =1$ if
$\lambda =-1$ and
$\varepsilon =\pm 1$
if $\lambda \ne -1$.
So each $\varepsilon
\sqrt{\lambda}\,
\Delta_n$ is either
$i\Delta_n$ or
$\pm\mu\Delta_n$, in
which
$\mu:=\sqrt{\lambda}$
and $\lambda=u+vi\ne
-1$. We can suppose
that $v\ge 0$ since
$\lambda$ is
determined up to
replacement by
$\bar{\lambda}$.
Because $\mu$ is
represented in the
form $a+bi$ with $a>
0$ and $b\ge 0$, the
equality
\begin{equation*}\label{hdu}
S_n\Delta_nS_n
=(-1)^n\Delta_n,
\qquad
S_n:=\diagg(j,-j,j,-j,\dots),
\end{equation*}
shows that we can
replace
$-\mu\Delta_n=(-a-bi)
\Delta_n$ by the
*congruent matrix
\[
S_n^*(-a-bi)\Delta_nS_n
=(-1)^n
S_n(-a-bi)\Delta_n
S_n=(-a+bi)\Delta_n
\]
and obtain the matrix
\eqref{gyo1}.

Now suppose that the
involution is
quaternionic
conjugation and $n$ be
even, or that the
involution is
quaternionic
semiconjugation and
$n$ is odd. Then $\bar
j=(-1)^{n+1}j$. By
\eqref{kki}, each
$\varepsilon
\sqrt{\lambda}\,
\Delta_n$ is either
$\Delta_n$ or
$\pm\mu\Delta_n$, in
which
$\mu:=\sqrt{\lambda}$
and $\lambda=u+vi\ne
1$ with $v\ge 0$.
Since $\mu$ is
represented in the
form $a+bi$ with $a\ge
0$ and $b>0$, we can
replace $-\mu\Delta_n
=(-a-bi)\Delta_n$ by
the *congruent matrix
\[
S_n^*(-a-bi)\Delta_nS_n
=(-1)^{n+1}
S_n(-a-bi)\Delta_n
S_n=(a-bi)\Delta_n
\]
and obtain the matrix
\eqref{gyo1}.

\end{document}